\documentclass[reqno]{amsart}
\usepackage[utf8]{inputenc}
\usepackage[T1]{fontenc}
\usepackage{amsmath,amssymb}
\usepackage{tikz}
\usetikzlibrary{patterns}
\theoremstyle{plain}
\newtheorem{theorem}{Theorem}[section]
\newtheorem{lemma}[theorem]{Lemma}
\newtheorem{proposition}[theorem]{Proposition}
\newtheorem{corollary}[theorem]{Corollary}
\theoremstyle{definition}
\newtheorem{definition}[theorem]{Definition}
\newtheorem{example}[theorem]{Example}
\theoremstyle{remark}
\newtheorem{remark}[theorem]{Remark}

\newcommand{\R}{\mathbb{R}}
\newcommand{\Sph}{\mathbb{S}}
\newcommand{\dist}{\mathrm{d}}
\newcommand{\Unp}{\mathrm{Unp}}
\newcommand{\reach}{\mathrm{reach}}
\newcommand{\polreach}{\mathrm{polreach}}
\newcommand{\conreach}{\mathrm{conreach}}
\newcommand{\Nor}{\mathrm{Nor}}
\newcommand{\Per}{\mathrm{Per}}
\newcommand{\vol}{\lambda_2}
\newcommand{\feet}{\Pi}

\begin{document}

\title{Conic reach and polynomial parallel volume in the plane}

\author{Alejandro Cholaquidis}
\address{Centro de Matem\'atica, Facultad de Ciencias, Universidad de la
Rep\'ublica, Igu\'a 4225, 11400 Montevideo, Uruguay}
\email{acholaquidis@cmat.edu.uy}

\subjclass[2020]{28A75, 52A38, 53C65, 62G05}
\keywords{Parallel volume, local Steiner formula, positive reach,
polynomial reach, conic singularities, set estimation}

\begin{abstract}
For a compact set $S\subset\R^2$, the local Steiner formula of Hug, Last
and Weil expresses the parallel volume $V_S(t)$ through the proximal
normal bundle and the truncated fiber lengths $\min\{t,\delta_S\}$. We
introduce conic reach, a geometric condition with two requirements: $S$
coincides with a cone near each point of a finite, well-separated singular
set, and every proximal normal fiber away from those points has length at
least $\rho$. In the plane, these requirements force every link to be a
finite union of circular arcs whose complementary gaps have width bounded
below. Computing the feet-localized tube of each cone and combining it
with the local Steiner formula, we show that $V_S$ is a polynomial of
degree at most two on $(0,\rho)$, with explicit coefficients;
in particular $\polreach(S)\ge\conreach(S)$. A one-dimensional converse
shows that a quadratic wall contribution forces a linear cut function.
Compact domains with piecewise-$C^2$ boundary, uniformly wedge-like at
their reentrant corners, have positive conic reach, and their volume
coefficients are given by a Gauss--Bonnet-type formula. For the L-shaped
polygon the three invariants separate: $\reach(L)=0$,
$\conreach(L)=\tfrac13$, $\polreach(L)=1$. A cuspidal notch, two
overlapping discs and a Cantor fan of segments show that tangential
contact, curvature at a reentrant corner and degenerating link gaps each
destroy polynomiality.
\end{abstract}

\maketitle

\section{Introduction}\label{sec:intro}

Let $S\subset\R^2$ be compact, let
$S_t:={y\in\R^2:\dist(y,S)\le t},$ $V_S(t):=\lambda_2(S_t)-\lambda_2(S),$
and call $V_S$ the parallel-volume function of $S$. For a planar convex body,
Steiner's formula gives $V_S(t)=\Per(S)t+\pi t^2,$ $t>0.$ Federer \cite{federer} extended this formula to sets of positive reach: on
$(0,\reach(S))$, the function $V_S$ agrees with a polynomial of degree at
most two, whose coefficients are curvature measures.

Positive reach excludes elementary nonconvex sets. An L-shaped polygon has
reach zero, since points on the exterior bisector of its reentrant corner
have two nearest points arbitrarily close to the boundary. Nevertheless,
$V_L(t)=8t+ (5\pi/4 -1)t^2$ 
for all sufficiently small $t$. Thus polynomial parallel volume near the
origin is strictly weaker than positive reach. Following \cite{CCM}, we
define the polynomial reach $\polreach(S)$ as the supremum of the $R\ge0$
such that $V_S$ agrees on $[0,R]$ with a polynomial of degree at most two;
see Definition~\ref{def:polreach}. The question addressed here is which
geometric conditions guarantee $\polreach(S)>0$ in the plane.

We introduce a conic certificate of radius $\rho$
(Definition~\ref{def:conreach}). Near each point $\sigma$ of a finite,
well-separated singular set, the set coincides exactly with a cone on a ball
of radius $3r_\sigma$; outside the corresponding cores, every proximal
normal fiber has length at least $\rho$. The supremum of all certifiable
values of $\rho$ is the conic reach $\conreach(S)$. The core radii depend on
the local geometry: sharper singularities require larger cores relative to
$\rho$; see Remark~\ref{rem:roles}.

No regularity of the cones is assumed. In the plane it follows from the
certificate itself. Theorem~\ref{thm:link} shows that every link is a finite
union of closed circular arcs and that each complementary gap has width at
least
$2\arctan (\rho/r_\sigma ).$
Thus finiteness and angular separation of the singular directions are
consequences of the definition.

The main result, Theorem~\ref{thm:main}, proves that a conic certificate of
radius $\rho$ yields $V_S(t)=c_1t+c_2t^2,$ $0<t<\rho,$
with both coefficients explicit. The linear coefficient is a boundary
length counted through proximal normals. The quadratic coefficient combines
the turning of the regular part, the normal fans at the singular points and
a correction $-\cot(\theta/2)$ for each link gap of width $\theta<\pi$. For
a reentrant corner, $\theta$ is its exterior opening. The same correction
appears in the Gauss--Bonnet formula of Theorem~\ref{thm:planarfree}, proved
independently for piecewise-$C^2$ domains with straight reentrant corners.
Remark~\ref{rem:gauss} compares the two formulas on the L-shaped polygon.

The proof uses a decomposition of the tube by nearest points. A spatial
cutoff is unsuitable: even for the half-line
$\mathcal R=\R_{\ge0}e_1$, intersecting its tube with a fixed ball creates a
cubic term. By contrast, the tube points whose feet lie within distance
$r_0$ of the endpoint have area exactly
$2r_0t+\frac{\pi}{2}t^2.$ 
The relevant object is therefore the length of each proximal normal fiber,
rather than global uniqueness of the metric projection. At a reentrant
corner the medial axis reaches the boundary and the fiber lengths tend to
zero, but exact conicity makes them proportional to the distance from the
vertex. This homogeneity is what preserves quadraticity.

A one-dimensional converse makes the mechanism precise. For a monotone
family of wall fibers with cut function $\delta$, an exactly quadratic
contribution
\[
\int_0^{r_0}\min\{t,\delta(a)\},da
\]
forces $\delta(a)=ca$ near the singular point
(Proposition~\ref{prop:converse}). Three examples show how polynomiality can
fail. A cuspidal notch produces a term of order $t^{3/2}$
(Proposition~\ref{prop:cusp}); a transversal but curved reentrant corner,
formed by two overlapping discs, gives a real-analytic non-polynomial volume
(Proposition~\ref{prop:twodiscs}); and a fan over the ternary Cantor set has
volume of order $t^{1-D}$, where $D=\log 2/\log 3$
(Proposition~\ref{prop:cantor}). These examples isolate tangential contact,
loss of one-homogeneity and degeneration of the link gaps as three distinct
obstructions.

\subsection{Relation to previous work}\label{sec:related}

The classical theory begins with Steiner's formula and Federer's extension
to sets of positive reach \cite{federer}; see \cite{RZbook} for a modern
account. Hug, Last and Weil \cite{HLW} proved a local Steiner formula for
arbitrary closed sets in terms of the proximal normal bundle. We use the
formulation of Khmaladze and Weil \cite{KW}, in which the contribution of a
normal pair $(z,u)$ is truncated at its fiber length $\delta_S(z,u)$. Conic
reach is designed to control these truncations even when the ordinary reach
vanishes.

General regularity of $t\mapsto V_S(t)$ is much weaker than exact
polynomiality. Differentiability of the parallel-volume function and its
relation with the surface area of parallel sets go back to Stach'o
\cite{Stacho}; see also Fu \cite{Fu85} and Rataj and Winter \cite{RW}. At the
opposite extreme, Heveling, Hug and Last \cite{HHL} proved that a planar
compact set whose parallel volume is polynomial on all of $(0,\infty)$ must
be convex. We only require polynomiality near the origin, which leaves room
for reentrant singularities. Related global characterizations for random
sets and lower-dimensional gauge bodies appear in \cite{HLWpoly}.

Polynomial parallel volume was used as a primitive assumption in the
statistical program of \cite{CCM}, where its coefficients and the maximal
interval of validity are estimated from random samples. The present results
provide a geometric sufficient condition for that assumption beyond
positive reach. The first coefficient,
$\lim_{t\downarrow0} {V_S(t)}/{t},$ 
is the unnormalized outer Minkowski content; its identification with
perimeter under suitable regularity assumptions is studied in \cite{ACV}.
In Theorem~\ref{thm:main} it is represented by
\[
\Theta_1(\eta)+2\sum_{\sigma\in\Sigma}m_\sigma r_\sigma,
\]
a boundary length recovered from the proximal-normal decomposition.

Conic reach should also be distinguished from the existence of curvature
measures or normal cycles. Compact domains with d.c. boundary \cite{PR} and
WDC sets \cite{PRZ} form broad classes of singular sets carrying such
structures; suitable finite unions are considered in \cite{PU,RZbook}.
These properties do not by themselves imply exact polynomiality of the
ordinary parallel volume. Non-polynomial tube behaviour already occurs for
locally finite unions of positive-reach sets under natural tangency
conditions \cite{CC}; Proposition~\ref{prop:twodiscs}, concerning the union
of two convex discs, gives an elementary example. The difference is that
conic reach controls the cutoff $\min{t,\delta_S(z,u)}$ appearing in the
local Steiner formula.

The Cantor-fan example is related to the theory of fractal tube formulas and
complex dimensions \cite{LvF}; the neighbourhood estimates used there are
classical \cite{Falconer}. Curvature measures for piecewise-flat spaces go
back to \cite{CMS}. Unlike positive-reach theories applied to the set or to
its complement, conic reach makes no positive-reach assumption on either
side of the boundary.

The rest of the paper is organized as follows.
Section~\ref{sec:motivation} gives the statistical motivation, and
Section~\ref{sec:back} introduces the proximal-normal and local-Steiner
machinery. Conic reach is defined in Section~\ref{sec:def}. The geometry of
the planar links and the localized cone tubes is developed in
Section~\ref{sec:link}, leading to the main theorem in
Section~\ref{sec:main}. Sections~\ref{sec:converse} and~\ref{sec:planar}
contain the one-dimensional converse and the explicit piecewise-$C^2$
formula, while Section~\ref{sec:sharp} presents the three failure
mechanisms. Technical lemmas and proofs are collected in
Appendices~\ref{app:tech} and~\ref{app:proofs}.

\section{Motivation: set estimation and the volume function}
\label{sec:motivation}

A recurring problem in nonparametric statistics is to reconstruct a compact
set $S\subset\R^2$, or geometric functionals of it, from a random sample.
The subject goes back to \cite{DW80}; general accounts of support estimation
are given in \cite{CF97,CFsurvey}. Such reconstruction requires geometric
assumptions that are weak enough to cover realistic sets and strong enough
to yield rates. Common examples include convexity and its relaxations,
rolling-ball and reach-type conditions \cite{CFP12}, and cone conditions
\cite{CCF14}. Federer's reach has become a standard regularity parameter in
this setting and can itself be estimated consistently \cite{CFM23}.

The parallel-volume function contains much of the relevant geometric
information. Its behaviour near $t=0$ governs the estimation of boundary
length and Minkowski content \cite{CFRC07,PLRC08,CFG13}. When $V_S$ is
polynomial on an interval $[0,R]$, its coefficients and the endpoint $R$
become natural statistical targets. This is the program of \cite{CCM}, where
the polynomial hypothesis is assumed and both the coefficients and the
polynomial reach are estimated from observations in $S$.

Positive reach is the classical sufficient condition for this hypothesis,
but it excludes sets that arise naturally in support estimation, such as
nonconvex polygons, crossings and unions of segments. Departures from
convexity are themselves objects of statistical inference \cite{CFMP24}.
The purpose of conic reach is to provide a checkable geometric condition that
covers such singularities while retaining an exact polynomial tube formula.

The coefficients also have a direct interpretation. The linear term is a
boundary length, while the quadratic term records the turning of the regular
boundary and the gap structure of the singular links. This suggests using
the volume function not only to estimate size and perimeter, but also to
extract information about the topology and singular geometry of a planar
set.

\section{Notation and background}\label{sec:back}

Throughout the paper, the ambient set $S\subset\R^2$ is assumed to be
nonempty and compact, unless explicitly stated otherwise. We
write $\lambda_2$ for two-dimensional and $\lambda_1$ for one-dimensional
Lebesgue measure.

Balls are $B(x,r)$ (open) and $\bar B(x,r)$; $\Sph^{1}$ is the unit sphere;
$\vol=\lambda_2$ is Lebesgue measure and $\mathcal H^k$ the Hausdorff measure. For
$S\subset\R^2$ closed, $y\in\R^2$:
\[
\feet_S(y):=\{x\in S: |y-x|=\dist(y,S)\}
\]
is the set of feet (nearest points) of $y$. Let
$\Unp(S):=\{y: \#\feet_S(y)=1\}$ and write $\xi_S(y)$ for the unique foot when it
exists. Federer's local reach is
\[
\reach(S,x):=\sup\{r\ge0: B(x,r)\subset\Unp(S)\},\quad
\reach(S)=\inf_{x\in S}\reach(S,x).
\]
For $x\in S$ and a unit vector $u$, the reach function (fiber length) is
\[
\delta_S(x,u):=\sup\bigl\{a>0: \xi_S(x+su)=x \text{for all }0<s<a\bigr\}
\in[0,\infty],
\]
with the convention $\sup\emptyset:=0$.
For $x$ in the interior of $S$, $\delta_S(x,\cdot)\equiv0$: for $0<\varepsilon<s$
small, the point $x+\varepsilon u\in S$ is strictly closer to $x+su$ than $x$
is, so $x$ is never a foot --- the normal bundle below therefore only sees
$\partial S$. The set of admissible $s$ is automatically an interval: if $\xi_S(x+su)=x$ then, for
$0<s'<s$, any $z\in S$ with $|x+s'u-z|\le s'$ would give
$|x+su-z|\le(s-s')+s'=s$ with equality only for $z=x$, so $\xi_S(x+s'u)=x$.
Consequently $\xi_S(x+su)=x$ forces $s\le\delta_S(x,u)$ --- the form of the
definition we will use in Theorem~\ref{thm:link}. We call $u$ a proximal
normal at $x$ if $\delta_S(x,u)>0$, write $\Nor(S,x)$ for the set of proximal unit
normals, and $N(S):=\{(x,u): x\in\partial S, \delta_S(x,u)>0\}$ 
for the (proximal) unit normal bundle; this matches the notation of \cite{HLW}.

For $u,\omega\in\Sph^1$ we write
$d_{\Sph^1}(u,\omega):=\arccos\langle u,\omega\rangle$ for the angular
distance, and $d_{\Sph^1}(u,K):=\inf_{\omega\in K}d_{\Sph^1}(u,\omega)$
for $K\subseteq\Sph^1$, with the convention
$d_{\Sph^1}(u,\emptyset):=+\infty$.

For $t>0$ let $W_t:=\{y: 0<\dist(y,S)\le t\},$ so that $V_S(t)=\vol(W_t);$
and for Borel sets $\Gamma\subseteq\partial S$ and $\eta\subseteq N(S)$, the
feet-localized tubes are
\[
\begin{aligned}
W_t(\Gamma)&:=\bigl\{y\in W_t\cap\Unp(S): \xi_S(y)\in\Gamma\bigr\},\\
W_t(\eta)&:=\Bigl\{y\in W_t\cap\Unp(S):\
\bigl(\xi_S(y),\tfrac{y-\xi_S(y)}{\dist(y,S)}\bigr)\in\eta\Bigr\}.
\end{aligned}
\]
Points with more than one foot are Lebesgue-null (Lemma~\ref{lem:federer}(i)
below), so the restriction to $\Unp(S)$ costs no volume.

\begin{definition}\label{def:polreach}
For $S\subset\R^2$ compact, let $\R_2[t]$ be the polynomials of degree at
most $2$ and define, following \cite{CCM},
\[
\polreach(S):=\sup\bigl\{R\ge 0: \exists p\in\R_2[t] \text{ with } V_S(t)=p(t)
 \forall t\in[0,R]\bigr\}.
\]
Since $V_S$ is continuous with $V_S(0)=0$, the closed interval and $p(0)=0$
cost nothing. (In the plane, the degree bound $\le2$ agrees with the one used in
\cite{CCM}.)
\end{definition}

The fiber length admits a useful characterization --- $\delta_S(x,u)$ is the
largest $s$ with $\dist(x+su,S)=s$ --- which also shows that $\delta_S$ is a
Borel function, upper semicontinuous on $S\times\Sph^1$; and by Federer's
theorem the points of $\R^2$ with more than one nearest point in $S$ form a
Lebesgue-null set, while the local reach bounds all proximal fiber lengths
from below. These three background facts are stated precisely as
Lemmas~\ref{lem:deltameas} and~\ref{lem:federer} in Appendix~\ref{app:tech} --- the reach function and the normal bundle
of a general closed set are part of the framework of \cite{HLW,KW}, and
the short proofs are included to keep the paper self-contained;
the first is elementary (cf. the reach function of \cite[Sect.~2]{KW}), and
the second is part of \cite[Thm.~4.8]{federer} together with the a.e.\
differentiability of semiconcave functions.

The regular zone is handled by the local Steiner formula for arbitrary
closed sets, which in the plane is particularly simple: two measures, no
normalization constants. Call a Borel set $\eta\subseteq N(S)$ $r$-bounded
if its feet lie in a bounded set and $\delta_S\ge\rho$ on $\eta$ for some
$\rho>0$.

\begin{theorem}\label{thm:HLW}
For every nonempty compact $S\subset\R^2$ there exist signed measures
$\Theta_0,\Theta_1$, defined and of finite total variation on the
$r$-bounded Borel subsets of $N(S)$, such that for every Borel
$\eta\subseteq N(S)$ whose feet lie in a compact set and every $t>0$,
\begin{equation}\label{eq:HLW}
\vol\bigl(W_t(\eta)\bigr)
=\int_\eta \min\{t,\delta_S\} d\Theta_1
+\tfrac12\int_\eta \min\{t,\delta_S\}^2 d\Theta_0 ,
\end{equation}
where for general $\eta$ the right-hand side is a finite weighted integral
against the variations of the $\Theta_j$.
\end{theorem}

\begin{remark}\label{rem:HLWattr}
Theorem~\ref{thm:HLW} is the planar case of \cite[Thm.~1]{KW}, applied to
the indicator of the localized tube --- measurable, since $\xi_S$ is Borel
on $\Unp(S)$; the finiteness is the integrability statement of that
theorem, and the uniqueness is also there. The formula goes back to
\cite[Thm.~2.1]{HLW}, stated with a different reach function; \cite{KW}
restate it with the corrected reach function
$r(x,u)=\sup\{s>0:\xi_S(x+su)=x\}$, which coincides with $\delta_S$ by
Lemma~\ref{lem:deltameas}. See also \cite{Santilli}.
\end{remark}

For sets with smooth boundary, $\Theta_1$ is the length measure of
$\partial S$ carried to $N(S)$ by the outer normal, and $\Theta_0$ is the
turning (curvature) measure; on the round disc of radius $r$,
$\Theta_1(N(S))=2\pi r$ and $\Theta_0(N(S))=2\pi$, and \eqref{eq:HLW}
returns $2\pi rt+\pi t^2$. Lower-dimensional pieces of $S$ are counted once
per proximal normal --- a segment of length $\ell$ has
$\Theta_1$-mass $2\ell$.

\section{Planar conic reach}\label{sec:def}

\begin{definition}\label{def:conreach}
Let $S\subset\R^2$ be compact and let $\rho>0$. We say that $S$ admits a
conic certificate of radius $\rho$ if there exist a finite set
$\Sigma\subset\partial S$, radii $r_\sigma\ge\rho$ ($\sigma\in\Sigma$) whose
enlarged cores are pairwise far apart,
\[
|\sigma-\sigma'| > 3 (r_\sigma+r_{\sigma'})
\qquad\text{for distinct }\sigma,\sigma'\in\Sigma,
\]
and cones $C_\sigma\subseteq\R^2$ (nonempty closed sets with
$\lambda C_\sigma=C_\sigma$ for all $\lambda>0$, not necessarily convex)
such that:
\begin{itemize}
\item[(C1)] the set is exactly conical on the enlarged core of each
singular point:
\[
S\cap\bar B(\sigma,3r_\sigma) = (\sigma+C_\sigma)\cap\bar B(\sigma,3r_\sigma),
\qquad \sigma\in\Sigma;
\]
\item[(C2)] outside the cores, every proximal normal fiber has length at
least $\rho$:
\[
\delta_S(z,u) \ge \rho
\qquad\text{for every }(z,u)\in N(S)\text{ with }
z\notin\bigcup_{\sigma\in\Sigma}B(\sigma,r_\sigma).
\]
\end{itemize}
The set $K_\sigma:=C_\sigma\cap\Sph^1$ is the link of $S$ at $\sigma$. The
conic reach of $S$ is the supremum of the certifiable radii,
\[
\conreach(S):=\sup\bigl\{\rho>0: S\text{ admits a conic certificate of
radius }\rho\bigr\}
\]
(with $\sup\emptyset:=0$).
\end{definition}

\begin{remark}\label{rem:roles}
The parameter $\rho$ is the quantity being maximized: the uniform fiber
length outside the singular cores, and the interval of polynomiality in
Theorem~\ref{thm:main}. The core radius $r_\sigma$ may vary from point to
point, and the two scales cannot be coupled: at a reentrant wedge of
exterior opening $\beta\in(0,\pi)$, the wall fiber at distance $a$ from
the vertex is cut by the exterior bisector at length $a\tan(\beta/2)$
(Figure~\ref{fig:linkfiber}), so the fiber condition forces
$r_\sigma\tan(\beta/2)\ge\rho$; coupling $r_\sigma=\rho$ would exclude
every corner with $\beta<\pi/2$. The factor $3$ in (C1) is a localization
margin: a point $y$ at distance less than $\rho$ from $S$ with a foot in
$B(\sigma,r_\sigma)$ satisfies $|y-\sigma|<r_\sigma+\rho$, and every
competing foot then lies within $r_\sigma+2\rho\le3r_\sigma$ of $\sigma$,
inside the region where $S$ and the cone agree (Lemma~\ref{lem:local},
Figure~\ref{fig:certificate}); the separation $3(r_\sigma+r_{\sigma'})$
keeps the conical zones disjoint. Condition (C2) concerns the proximal
normal bundle only. A uniform unique-projection requirement in its place
would fail: the exterior bisector of a wedge contains two-feet points
arbitrarily close to the vertex. Imposed fiberwise, (C2) instead controls
directly the cutoff $\min\{t,\delta_S\}$ of the localized Steiner formula:
outside the cores the cut locus does not truncate a fiber before length
$\rho$, while inside them conicity makes the cut lengths one-homogeneous
in the distance to the vertex, and homogeneous integrands integrate to
polynomials. When $\Sigma=\emptyset$ the condition is a global fiber
bound.
\end{remark}

\begin{figure}[t]
\centering
\begin{tikzpicture}[scale=0.7]
\fill[black!10] (-4,1.6) -- (0,1.6) -- (0,0) -- (4,0) -- (4,-1.6) -- (-4,-1.6) -- cycle;
\draw[thick] (-4,1.6) -- (0,1.6) (0,1.6) -- (0,0) -- (4,0);
\fill (0,0) circle (1.6pt);
\node[below left] at (0,0) {$\sigma$};
\draw[dashed] (0,0) circle (1);
\draw[dashed] (0,0) circle (3);
\node at (-0.55,1.22) {\small $r_\sigma$};
\node at (-2.35,2.35) {\small $3r_\sigma$};
\draw[dotted] (0,0) -- (3.4,3.4);
\fill (0.6,0) circle (1.3pt); \node[below] at (0.6,0) {\small $z$};
\draw (0.6,0) -- (0.6,0.52);
\draw (0.6,0.52) circle (1.3pt); \node[above right] at (0.6,0.5) {\small $y$};
\draw[densely dotted] (0.6,0.52) -- (0,0.62);
\fill (0,0.62) circle (1.3pt); \node[left] at (0,0.62) {\small $q$};
\fill (2.2,0) circle (1.3pt);
\draw[->,thick] (2.2,0) -- (2.2,1.5) node[right] {\small $\delta\ge\rho$};
\end{tikzpicture}
\caption{Localization near a singular core (schematic): the set (shaded)
coincides with the cone $\sigma+C_\sigma$ on $\bar B(\sigma,3r_\sigma)$.
If a foot $z$ lies in the core $B(\sigma,r_\sigma)$ and
$\dist(y,S)<\rho\le r_\sigma$, then $y$ and every competing foot $q$ lie
in $\bar B(\sigma,3r_\sigma)$, where $S$ and the cone agree
(Lemma~\ref{lem:local}); fibers over feet outside the cores have length
at least $\rho$.}
\label{fig:certificate}
\end{figure}
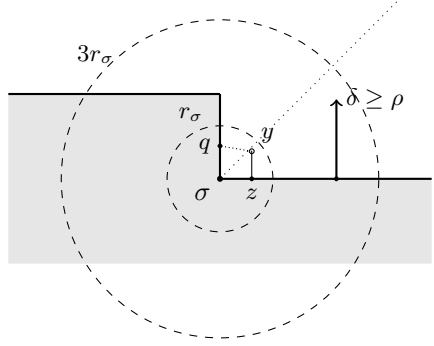

Nothing is assumed about the cones $C_\sigma$: Theorem~\ref{thm:link}
will show that (C1)--(C2) already force every link to be a finite union
of closed arcs with uniformly wide gaps.

\begin{figure}[t]
\centering
\begin{tikzpicture}[scale=0.72]
\draw (0,0) circle (1.6);
\draw[line width=2.2pt] (35:1.6) arc (35:130:1.6);
\draw[line width=2.2pt] (200:1.6) arc (200:325:1.6);
\fill (-35:1.6) circle (1.7pt); \node[right] at (-35:1.85) {\small $\omega_1$};
\fill (35:1.6) circle (1.7pt); \node[right] at (35:1.85) {\small $\omega_2$};
\draw[<->] (-27:1.95) arc (-27:27:1.95);
\node at (0:2.3) {\small $\theta$};
\node at (100:2.0) {\small $K$};
\node at (0:1.15) {\small $G$};
\node at (0,-2.6) {\small the link and one gap};
\begin{scope}[xshift=7.2cm]
\fill[black!10] (35:2.6) arc (35:325:2.6) -- (0,0) -- cycle;
\draw[thick] (0,0) -- (-35:2.6);
\draw[thick] (0,0) -- (35:2.6);
\fill (0,0) circle (1.6pt);
\draw[dotted] (0,0) -- (2.75,0);
\node[above] at (1.65,0.06) {\small bisector};
\fill (-35:2.1) circle (1.5pt);
\draw[thin] (1.589,-1.113) -- (1.681,-0.981) -- (1.812,-1.073);
\node[below] at (-35:2.32) {\small $r\omega_1$};
\draw[->,thick] (-35:2.1) -- (2.564,0);
\node[right] at (2.30,-0.92) {\small $\delta=r\tan\frac\theta2$};
\node at (0,-2.6) {\small the fiber over a boundary ray};
\end{scope}
\end{tikzpicture}
\caption{Left: a link $K\subset\Sph^1$ (thick arcs) with a gap $G$ of
width $\theta$ between the endpoints $\omega_1,\omega_2$. Right: in the
cone over $K$ (shaded), the fiber over $r\omega_1$ in the direction of the
gap is cut by the gap's bisector (dotted) at length $r\tan(\theta/2)$; for
$\theta\ge\pi$ the fiber is infinite (Lemma~\ref{lem:sidefiber}).}
\label{fig:linkfiber}
\end{figure}

\begin{remark}\label{rem:reachincl}
With $\Sigma=\emptyset$, a certificate of radius $\rho$ is the requirement
$\delta_S\ge\rho$ on all of $N(S)$; since $\reach(S,z)\ge\rho$ gives
$\delta_S(z,u)\ge\rho$ for all $(z,u)\in N(S)$
(Lemma~\ref{lem:federer}(ii)), every $\rho<\reach(S)$ is certifiable and
$\conreach(S)\ge\reach(S)$. The inclusion of classes is strict: the
L-shape has $\reach=0$ and $\conreach=\tfrac13$ (Example~\ref{ex:Lshape}).
Moreover, for $a>0$, $Q$ orthogonal and $b\in\R^2$,
$\conreach(aQS+b)=a \conreach(S)$: a certificate
$(\rho,\Sigma,(r_\sigma),(C_\sigma))$ for $S$ becomes the certificate
$(a\rho, aQ\Sigma+b, (ar_\sigma), (QC_\sigma))$ for $aQS+b$, since
$\delta_{aQS+b}(aQz+b,Qu)=a \delta_S(z,u)$. The same invariance holds for
$\reach$ and, since $V_{aQS+b}(t)=a^2 V_S(t/a)$, for $\polreach$.
\end{remark}

\begin{remark}\label{rem:convexcorners}
Convex corners need not be included in $\Sigma$: where two $C^2$ boundary
pieces meet transversally with interior angle $<\pi$, the exterior normal
directions spread at the corner, fibers over nearby boundary points are
cut by curvature or by far-away features rather than by the corner itself,
and (C2) holds there with no conical hypothesis --- indeed, when the
pieces are curved, (C1) forbids putting the point in $\Sigma$. (General
prox-regularity statements for transversal intersections of positive-reach
sets exist, see \cite[Ch.~5]{RZbook}, but we make no use of them: (C2) is
verified directly in every example, and the planar case is settled in
Section~\ref{sec:planar}.) Exact conicity is thus required only where the
cut locus accumulates on $\partial S$ (reentrant corners, crossings, cone
points).
\end{remark}

\begin{example}\label{ex:Lshape}
The L-shape $L=([0,2]\times[0,1])\cup([0,1]\times[0,2])$ has
\[
\conreach(L)=\tfrac13,
\qquad\text{and, by Remark~\ref{rem:reachincl},}\qquad
\conreach(aL)=\tfrac a3  (a>0),
\]
and the supremum is attained. Lower bound: take $\Sigma=\{q\}$ with
$q=(1,1)$ the reentrant corner, $C_q=\{(x,y):x\le0\text{ or }y\le0\}$ the
solid three-quarter-plane wedge (one link component, one gap of width
$\tfrac\pi2$), and $r_q=\rho=\tfrac13$. Conicity holds on
$\bar B(q,3r_q)=\bar B(q,1)$: a point of that ball lies in $L$ exactly
when one of its coordinates is $\le1$, i.e. exactly when it lies in
$q+C_q$, and radius $1$ is precisely where this stops --- beyond it,
$q+C_q$ contains points below the floor $y=0$, and the ends of the arms
come into view. As for the fibers outside $B(q,\tfrac13)$: over an
interior wall point $z$ at distance $a=|z-q|\ge\tfrac13$ from the corner,
the fiber is cut by the exterior bisector at
$\delta=a\tan\tfrac\pi4=|z-q|\ge\tfrac13$; over the outer sides the
relevant normals are support normals with infinite fibers; at the three
outer corners $(0,0)$, $(2,0)$, $(0,2)$, every direction $u$ of the normal
fan is a global support normal ($\langle\cdot,u\rangle$ is maximized over
$L$ at the corner), so those fibers are infinite as well; and at $(2,1)$
and $(1,2)$, the only competing boundary is the other arm, which stays at
distance $\ge1$ from every point of the fan fibers, so these are not cut
before length $1$. Hence the certificate is admissible with
$\rho=\tfrac13$. Upper bound: let a certificate of radius $\rho>0$ be given. First,
$q\in\bar B(\sigma,r_\sigma)$ for some $\sigma\in\Sigma$: otherwise,
$\Sigma$ being finite, some $\varepsilon>0$ would give
$B(q,\varepsilon)\cap\bigcup_\sigma B(\sigma,r_\sigma)=\emptyset$, and
interior wall points $z\to q$ would lie outside all cores while carrying
fibers of length exactly $|z-q|\to0$, contradicting (C2). Next,
$\sigma=q$: the point $q$ lies in the open ball $B(\sigma,3r_\sigma)$,
where $\partial S$ and $\partial(\sigma+C_\sigma)$ coincide, and
$\partial(\sigma+C_\sigma)=\sigma+\partial C_\sigma$ is invariant under
dilations centred at $\sigma$, hence contains the full ray from $\sigma$
through each of its points. Near $q$, $\partial L$ contains two
non-collinear segments through $q$; if $\sigma$ were off the line of one
of them, the rays from $\sigma$ through that segment would sweep a set
with nonempty interior inside $\partial(\sigma+C_\sigma)$ --- impossible,
since $C_\sigma$ is closed, and the boundary of a closed set has empty
interior. So $\sigma$ lies on
both lines, whose intersection is $\{q\}$. Finally, a cone is determined
by its germ at the apex
($C=\bigcup_{\lambda>0}\lambda (C\cap B_\varepsilon)$), so the cone of
the certificate is the wedge $C_q$, and (C1) on $\bar B(q,3r_q)$ forces
$3r_q\le1$: beyond radius $1$ the wedge contains points below the floor
$y=0$ that are not in $L$. As $\rho\le r_q$ by definition,
$\rho\le\tfrac13$.
\end{example}

\section{Planar links and the localized cone tube}\label{sec:link}

\subsection{Structure of the links}

The next theorem is the reason planar conic reach needs no regularity
clause: the fiber condition (C2), transported to the cone by
Corollary~\ref{cor:trunc}, computes the boundary fibers of every link
exactly and forces all gaps to be wide. Throughout, for a cone
$C\subseteq\R^2$, a unit $\omega\in C$ and a unit $u$, we write
$\tilde\delta(\omega,u):=\delta_C(\omega,u)\in[0,\infty]$; by homogeneity,
$\delta_C(r\omega,u)=r \tilde\delta(\omega,u)$ and
$\Nor(C,r\omega)=\Nor(C,\omega)$ for $r>0$, and $\tilde\delta(\omega,u)>0$
exactly when $u$ is proximal at $\omega$.

\begin{theorem}\label{thm:link}
Let $S$ admit a conic certificate of radius $\rho$ and let
$\sigma\in\Sigma$. Then $K_\sigma\ne\Sph^1$, every gap of $K_\sigma$ has
angular width at least $\theta_\sigma:=2\arctan\frac{\rho}{r_\sigma},$ 
and consequently $K_\sigma$ is the union of at most
$\lfloor 2\pi/\theta_\sigma\rfloor$ pairwise disjoint closed arcs, some
possibly degenerate to points. (The case $K_\sigma=\emptyset$, i.e.
$S\cap\bar B(\sigma,3r_\sigma)=\{\sigma\}$, is allowed.)
\end{theorem}

The proof, in Appendix~\ref{app:pflink}, transports the fiber condition
(C2) to the cone by exact localization and computes the boundary fibers of
the link explicitly.

\subsection{The tube of a planar cone}\label{sec:cone}

Once the links are finite unions of arcs, the feet-localized tube of a
cone is a direct computation, carried out in the isometric coordinates
$(r,s)\mapsto r\omega+su$ attached to each incidence of a gap and one of
its endpoints (Figure~\ref{fig:conetube}).

\begin{figure}[t]
\centering
\begin{tikzpicture}[scale=1.0]
\fill[black!12] (0,0) -- (1.67,1.5) -- (4,1.5) -- (4,0) -- cycle;
\draw[->] (0,0) -- (4.6,0) node[below] {\small $r$};
\draw[->] (0,0) -- (0,2.3) node[left] {\small $s$};
\draw (0,0) -- (2.3,2.07);
\node[rotate=42] at (1.05,1.28) {\small $s=r \Delta(\theta_i)$};
\draw[dashed] (0,1.5) -- (4,1.5);
\node[left] at (0,1.5) {\small $t$};
\draw (4,0.06) -- (4,-0.06); \node[below] at (4,-0.06) {\small $r_0$};
\draw[dotted] (1.67,0) -- (1.67,1.5);
\draw (1.67,0.06) -- (1.67,-0.06);
\node[below] at (1.67,-0.1) {\small $t\cot\frac{\theta_i}2$};
\node at (2.75,0.72) {\small $\min\{t,r\Delta(\theta_i)\}$};
\end{tikzpicture}
\caption{The tube region of one incidence (gap, endpoint) in the
isometric coordinates $(r,s)\mapsto r\omega+su$ of
Proposition~\ref{prop:cone}: feet at radius $r$ carry fibers truncated at
$\min\{t,r\Delta(\theta_i)\}$ (shaded). Its area is
$\int_0^{r_0}\min\{t,r\Delta(\theta_i)\} dr
=r_0t-\tfrac12\cot(\theta_i/2) t^2$ for $\theta_i<\pi$; the vertex fan
contributes $\gamma_C t^2$.}
\label{fig:conetube}
\end{figure}
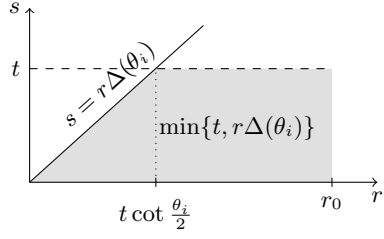

\begin{proposition}\label{prop:cone}
Let $C\subseteq\R^2$ be a cone whose link $K$ is a finite union of $m\ge0$
pairwise disjoint closed arcs (possibly points), $K\ne\Sph^1$, and let
$r_0>0$. List the gaps of $K$ as $G_1,\dots,G_m$ --- for a closed
$K\ne\emptyset,\Sph^1$, the number of complementary gaps equals the number
$m$ of connected components of $K$; if $K=\emptyset$, set $m=0$ --- with
widths $\theta_1,\dots,\theta_m$. Write
\[
\Delta(\theta):=\begin{cases}\tan(\theta/2), & 0<\theta<\pi,\\
+\infty, & \pi\le\theta\le2\pi,\end{cases}
\]
for the normalized fiber length of Lemma~\ref{lem:sidefiber}, and put
$\Delta_{\min}:=\min_i\Delta(\theta_i)$ (with $\min\emptyset:=+\infty$, and
$1/\Delta:=0$ when $\Delta=+\infty$). Then, for
\[
T_{C,r_0}(t):=\lambda_2\bigl(\{y\in\Unp(C): 0<\dist(y,C)\le t,\
\xi_C(y)\in\bar B(0,r_0)\}\bigr)
\]
and every $0<t\le r_0\Delta_{\min}$,
\[
T_{C,r_0}(t)=2m r_0t+
\Big(\gamma_C-\sum_{i: \theta_i<\pi}\cot\tfrac{\theta_i}{2}\Big)t^2,
\]
where $\gamma_C:=\tfrac12 \mathcal H^1\bigl(\{u\in\Sph^1:\
d_{\Sph^1}(u,K)\ge\tfrac\pi2\}\bigr),$ with $\gamma_C=\pi$ when $K=\emptyset$.
\end{proposition}

The proof, by integration of the side fibers of
Lemma~\ref{lem:sidefiber} and of the vertex fan, is in
Appendix~\ref{app:pfmain}.

\section{The main theorem}\label{sec:main}

\begin{theorem}\label{thm:main}
Let $S\subset\R^2$ be compact, admitting a conic certificate of radius
$\rho$ with data $\Sigma$, $(r_\sigma)$, $(C_\sigma)$, and write
\[
\Gamma_{\mathrm{reg}}:=\Bigl\{z\in\partial S: z\notin
\bigcup_{\sigma\in\Sigma}B(\sigma,r_\sigma)\Bigr\},\qquad
\eta:=N(S)\cap\bigl(\Gamma_{\mathrm{reg}}\times\Sph^1\bigr).
\]
For $\sigma\in\Sigma$ let $m_\sigma$ be the number of connected components of
the link $K_\sigma$ (finite by Theorem~\ref{thm:link}), let the gaps of
$K_\sigma$ be the connected components $G$ of $\Sph^1\setminus K_\sigma$,
and let $\gamma_\sigma:=\tfrac12 \mathcal H^1 (\{u\in\Sph^1:\
d_{\Sph^1}(u,K_\sigma)\ge\tfrac\pi2\})$
be the fan constant of $\sigma$. Convention for $K_\sigma=\emptyset$ (an
isolated point of $S$): $m_\sigma:=0$, there are no gaps, and $\gamma_\sigma=\pi$, as in
Proposition~\ref{prop:cone}. Then, for all
$0<t<\rho$,
\begin{equation}\label{eq:mainformula}
V_S(t) = \Bigl[\Theta_1(\eta)+2\sum_{\sigma\in\Sigma}m_\sigma r_\sigma\Bigr] t
 + \Bigl[\tfrac12 \Theta_0(\eta)
+\sum_{\sigma\in\Sigma}\Bigl(\gamma_\sigma
-\sum_{\substack{G \mathrm{gap of} K_\sigma\ |G|<\pi}}
\cot\tfrac{|G|}{2}\Bigr)\Bigr] t^2 ,
\end{equation}
where $\Theta_0,\Theta_1$ are the planar reach measures of
Theorem~\ref{thm:HLW} and $|G|$ denotes the angular width of a gap. In
particular $V_S$ is a polynomial of degree at most $2$ vanishing at $0$ on
$(0,\rho)$, and $\polreach(S) \ge \conreach(S).$
\end{theorem}

The right-hand side of \eqref{eq:mainformula} depends on the chosen
certificate $(\rho,\Sigma,(r_\sigma))$, while the left-hand side does not:
different certificates redistribute the same coefficients between the regular
and the conical terms (Remark~\ref{rem:gauss} shows the mechanism on the
L-shape). The proof, in Appendix~\ref{app:pfmain}, partitions the tube by the
location of the nearest point: the regular contribution follows from the
local Steiner formula, while each singular contribution is identified with
the localized tube of the corresponding cone (Appendix~\ref{app:loclem}).

\begin{remark}\label{rem:gauss}
The certificate-dependence cancels in the sum. For the L-shape
$L=([0,2]\times[0,1])\cup([0,1]\times[0,2])$ with the certificate of
Example~\ref{ex:Lshape} ($\Sigma=\{q\}$, $q=(1,1)$, any
$0<\rho=r_q\le\tfrac13$), the link is a single arc of width
$\tfrac{3\pi}2$, with one gap of width $\tfrac\pi2$ and empty fan: the
corner contributes $2r_q t$ to the linear coefficient and
$-\cot\tfrac\pi4=-1$ to the quadratic one. The regular zone carries
$\Theta_1(\eta)=8-2r_q$ (the perimeter minus the two wall segments hidden
in $B(q,r_q)$) and turning $\tfrac12\Theta_0(\eta)=\tfrac{5\pi}4$ from the
five convex right angles. The linear coefficient is
$(8-2r_q)+2r_q=8=\Per(L)$, the $r_q$-dependence cancelling, and the
quadratic one is $\tfrac{5\pi}4-1$, in agreement with
Proposition~\ref{prop:planar} below ($\beta_v=$ exterior opening $=$ gap
width): the decomposition depends on the certificate, but the polynomial
does not. A direct partition of the tube by feet gives the same polynomial
on the larger interval $(0,1)$.
\end{remark}

\begin{proposition}\label{prop:polreachL}
$\polreach(L)=1$. The L-shape thus realizes the exact chain
\[
\reach(L)=0 < \conreach(L)=\tfrac13 < \polreach(L)=1 .
\]
\end{proposition}

The proof, an elementary expansion of the fan overlap at $t=1$, is in
Appendix~\ref{app:pfL}.

\section{A one-dimensional converse for the cut function}\label{sec:converse}

Theorem~\ref{thm:main} is a sufficiency statement, and it is natural to ask
how far exact conicity is from being necessary. At the level of a single monotone family of wall fibers one obtains a
complete one-dimensional characterization, because the truncated integral
remembers the inverse of the cut function.

\begin{proposition}\label{prop:converse}
Let $r_0>0$, let $\delta:[0,r_0]\to[0,\infty)$ be continuous and strictly
increasing with $\delta(0)=0$, and set
\[
I_\delta(t):=\int_0^{r_0}\min\{t,\delta(a)\} da,
\qquad 0<t<\delta(r_0).
\]
Then $I_\delta$ is differentiable on $(0,\delta(r_0))$, with
$I_\delta'(t) = r_0-\delta^{-1}(t).$ 
In particular, $I_\delta$ agrees with a polynomial of degree at most two on
some interval $(0,\varepsilon)$ if and only if $\delta(a)=ca$ on some
interval $(0,\varepsilon')$, for a constant $c>0$.
\end{proposition}

\begin{proof}
Write $a(t):=\delta^{-1}(t)$, a continuous increasing function on
$[0,\delta(r_0)]$ with $a(0)=0$. For $0<t<t+h<\delta(r_0)$,
\[
I_\delta(t+h)-I_\delta(t)
=\int_0^{r_0}\bigl(\min\{t+h,\delta(a)\}-\min\{t,\delta(a)\}\bigr) da .
\]
The integrand vanishes for $a\le a(t)$, equals $h$ for $a\ge a(t+h)$, and
lies in $[0,h]$ on the transition interval $\bigl(a(t),a(t+h)\bigr)$, whose
length tends to $0$ as $h\downarrow0$ by continuity of $a$. Hence
$I_\delta(t+h)-I_\delta(t)=h (r_0-a(t))+o(h)$, and the analogous
computation for $h\to0^-$ gives the same left derivative: $I_\delta'(t)
=r_0-a(t)$, a continuous function of $t$. If $I_\delta$ agrees with a
polynomial of degree at most two on $(0,\varepsilon)$, then
$a=\delta^{-1}$ is affine there; since $a(0^+)=0$ and $a$ is increasing and
not constant, $a(t)=t/c$ on $(0,\varepsilon)$ for some $c>0$, i.e.\
$\delta(a)=ca$ for $0<a<\varepsilon/c$. Conversely, if $\delta(a)=ca$ near
$0$ then $I_\delta'(t)=r_0-t/c$ near $0$, and $I_\delta(0^+)=0$ gives
$I_\delta(t)=r_0t-t^2/(2c)$ there.
\end{proof}

\begin{corollary}\label{cor:converse}
Suppose that on some interval $(0,\varepsilon)$, with
$0<\varepsilon\le\delta(r_0)$, the parallel volume
decomposes as $V_S(t)=p(t)+I_\delta(t)$, where $p$ is a polynomial of
degree at most two and $\delta$ is as in
Proposition~\ref{prop:converse}. If $V_S$ is a polynomial of degree at
most two on $(0,\varepsilon)$, then $\delta(a)=ca$ near $0$ for some
$c>0$.
\end{corollary}

\begin{proof}
$I_\delta=V_S-p$ is a polynomial of degree at most two on
$(0,\varepsilon)$; apply Proposition~\ref{prop:converse}.
\end{proof}

The hypothesis that the remaining localized contributions are polynomial
is the precise form of ``no cancellation'' between singularities.

\begin{remark}\label{rem:cancel}
For the contribution of a single wall with a monotone cut function, the
proposition upgrades one-homogeneity from a sufficient device to a
necessary mechanism: an exactly quadratic wall term forces
$\delta(a)=ca$, which is what exact conicity produces
(Theorem~\ref{thm:link} and \eqref{eq:fiberwedge}). What the proposition
does not rule out is cancellation between several families: if two
singularities contribute inverse cut functions
$\delta_1^{-1}(t)=a_1t+bt^2$ and $\delta_2^{-1}(t)=a_2t-bt^2$ with
$a_1,a_2,b>0$, for $t>0$ small enough that both remain increasing, the sum of the two wall terms is a polynomial while neither
cut function is linear. Realizing such a pair geometrically, without
introducing further non-polynomial contributions, is precisely what
Remark~\ref{rem:strict} asks for.
\end{remark}

\begin{remark}\label{rem:strict}
We do not know whether the inclusion
$\{\conreach>0\}\subseteq\{\polreach>0\}$ is strict in the plane. In the
examples of Section~\ref{sec:sharp}, tangential contact, curvature at a
reentrant corner and degenerating link gaps each destroy conic reach and
polynomiality together, and by \cite{HHL} a separating set could be
polynomial only locally. A candidate would need the cancellation of
Remark~\ref{rem:cancel}: two singularities whose inverse cut functions
have opposite quadratic terms.
\end{remark}

\section{An explicit formula for piecewise-$C^2$ boundaries}\label{sec:planar}

For piecewise-$C^2$ planar domains the coefficients can be expressed in
terms of the curvature of the regular arcs and the corner angles, with
all constants explicit. Throughout this section $S\subset\R^2$ is compact, $S=\overline{\operatorname{int}S}$,
and $\partial S$ is a finite disjoint union of simple closed curves
(Jordan curves), each a finite concatenation of $C^2$ arcs meeting at corners with interior angles
$\theta_v\in(0,2\pi)$, $\theta_v\neq\pi$ (the interior angle is measured inside $S$).
By a $C^2$ arc we mean a regular $C^2$ embedding of a compact interval,
$C^2$ up to its endpoints, so that tangents, normals and second-order
Taylor expansions exist at the endpoints; distinct arcs meet only at the
common endpoints prescribed by the concatenation. We write
$\Gamma_1,\dots,\Gamma_M$ for the closed arcs of the decomposition.
Corners with $\theta_v<\pi$ are convex, those with $\theta_v>\pi$ are
reentrant. Write $\partial S_{\mathrm{reg}}$ for $\partial S$ minus the
corners, and let $\kappa$ be the signed curvature of the arcs, normalized through
the outward unit normal $n$ and an arc-length parametrization by
$n'(s)=\kappa(s) \tau(s)$ ($\tau$ the unit tangent), so that the exterior normal
map $z\mapsto z+sn(z)$ has Jacobian $1+s\kappa$. (Orient each component so that the exterior normal is the left normal of
$\tau$; then $\kappa$ is well defined, the outer boundary of a disc is
traversed clockwise, and a hole boundary counterclockwise.) A disc of radius
$r$ has $\kappa\equiv1/r$ and $\int_{\partial S_{\mathrm{reg}}}\kappa ds=2\pi$; the
boundary of an interior hole contributes with the opposite sign.

\begin{definition}\label{def:wedgelike}
A solid wedge of interior opening $\theta\in(0,2\pi)$ is, up to a rigid
motion, the closed cone
$\{r(\cos\phi,\sin\phi): r\ge0, 0\le\phi\le\theta\}$. We say that $S$
is uniformly wedge-like at its reentrant corners, with parameter $\ell>0$,
if at every reentrant corner $v$ the set coincides with a solid wedge
$v+C_v$ of interior opening $\theta_v$ on the ball of radius $\ell$,
$S\cap\bar B(v,\ell)=(v+C_v)\cap\bar B(v,\ell)$ 
--- in particular the two incident arcs are straight there, and no other
part of $\partial S$ enters $\bar B(v,\ell)$ --- and the closed
$\ell/2$-neighbourhoods of distinct reentrant corners are disjoint. The
exterior opening of the wedge at $v$ is $\beta_v:=2\pi-\theta_v\in(0,\pi)$.
\end{definition}

\begin{theorem}\label{thm:planarfree}
Every compact $S=\overline{\operatorname{int}S}$ whose boundary is a
finite disjoint union of piecewise-$C^2$ Jordan curves, and which is
uniformly wedge-like at its reentrant corners
(Definition~\ref{def:wedgelike}), satisfies, for some $t_0>0$ and all
$0<t<t_0$,
\begin{equation}\label{eq:planar}
V_S(t)=\Per(S) t+c_2 t^2,\quad
c_2=\frac12\int_{\partial S_{\mathrm{reg}}}\kappa ds
+\sum_{v: \theta_v<\pi}\frac{\pi-\theta_v}{2}
+\sum_{v: \theta_v>\pi}\cot\frac{\theta_v}{2},
\end{equation}
where $\Per(S)=\mathcal H^1(\partial S)$. Equivalently, by Gauss--Bonnet, with
$\chi(S)$ the Euler characteristic,
\begin{equation}\label{eq:planarGB}
c_2=\pi \chi(S)+\sum_{v: \theta_v>\pi}\psi(\theta_v),\quad
\psi(\theta):=\cot\frac{\theta}{2}+\frac{\theta-\pi}{2}  (<0 \text{for} \theta>\pi).
\end{equation}
Only the reentrant corners produce a correction; convex corners are already accounted
for by the Gauss--Bonnet bookkeeping.
\end{theorem}

The proof occupies the rest of the section:
Proposition~\ref{prop:planar} establishes the expansion, with an explicit
admissible $t_0$, under (P1) together with a uniform fiber bound (P2), and
Proposition~\ref{prop:p2auto} shows that the fiber bound is automatic in
this class. Figure~\ref{fig:planarlocal} shows the three local
contributions to \eqref{eq:planar}.

\begin{figure}[t]
\centering
\begin{tikzpicture}[scale=0.56]
\fill[black!10] (-2,-1.5) rectangle (2,0);
\fill[black!10] (0,0) ++(-2,0) arc (180:0:2) -- (2,0) -- cycle;
\draw[thick] (-2,0) arc (180:0:2);
\foreach \a in {35,75,115,155} {\draw[->] (\a:2) -- (\a:2.62);}
\node at (0,3.1) {\small $t+\tfrac12\kappa t^2$ per unit length};
\node at (0,-2.3) {\small (a) regular arc};
\begin{scope}[xshift=7.0cm]
\fill[black!10] (0,0) -- (2.1,0) -- (2.1,2.1) -- (0,2.1) -- cycle;
\draw[thick] (2.1,0) -- (0,0) -- (0,2.1);
\fill[black!28] (0,0) -- (-1.15,0) arc (180:270:1.15) -- cycle;
\draw (0,0) -- (-1.15,0) (0,0) -- (0,-1.15);
\node at (-1.35,-1.35) {\small $\tfrac12(\pi-\theta_v) t^2$};
\draw (0.55,0) arc (0:90:0.55); \node at (45:0.88) {\small $\theta_v$};
\node at (0.5,-2.3) {\small (b) convex corner};
\end{scope}
\begin{scope}[xshift=14.0cm]
\fill[black!10] (35:2.5) arc (35:325:2.5) -- (0,0) -- cycle;
\draw[thick] (0,0) -- (-35:2.5); \draw[thick] (0,0) -- (35:2.5);
\draw[dotted] (0,0) -- (2.7,0);
\fill[pattern=north east lines] (0,0) -- (1.22,0) -- (0.401,0.573) -- cycle;
\draw[thin] (0.401,0.573) -- (1.22,0);
\draw[->] (-35:1.0) -- (1.221,-0.001);
\draw[->] (-35:2.0) -- (2.040,-0.575);
\node at (2.62,0.72) {\small deficit $\tfrac12\cot\tfrac{\beta_v}2 t^2$};
\draw (0.62,0) arc (0:-35:0.62); \node at (-16:1.0) {\small $\tfrac{\beta_v}2$};
\node at (0.4,-2.9) {\small (c) reentrant corner};
\end{scope}
\end{tikzpicture}
\caption{The three local contributions to the quadratic coefficient in
\eqref{eq:planar}. (a) A regular arc contributes its turning,
$\tfrac12\int\kappa ds$. (b) A convex corner contributes the area of its
exterior normal fan, of opening $\pi-\theta_v$. (c) At a reentrant corner
the exterior bisector (dotted) truncates the wall fibers near the vertex;
the hatched deficit on each wall equals $\tfrac12\cot(\beta_v/2) t^2$,
giving the correction $\cot(\theta_v/2) t^2<0$.}
\label{fig:planarlocal}
\end{figure}
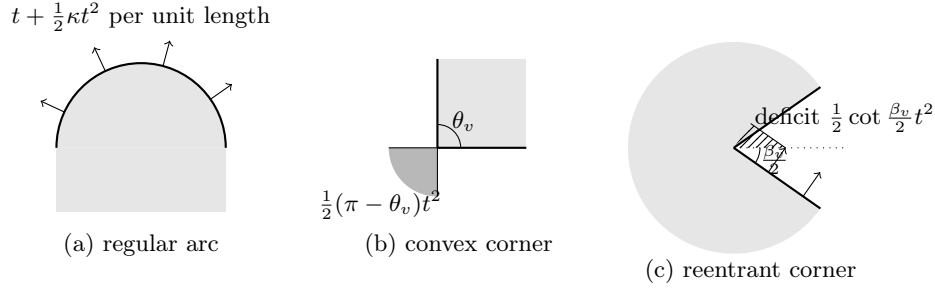

\begin{proposition}\label{prop:planar}
Assume in addition that
\begin{itemize}
\item[(P1)] $S$ is uniformly wedge-like at its reentrant corners, with
parameter $\ell>0$ (Definition~\ref{def:wedgelike});
\item[(P2)] there is $\rho>0$ such that
\[
\delta_S(z,u) \ge \rho
\qquad\text{for every }(z,u)\in N(S)\text{ with }
z\notin\bigcup_{v: \theta_v>\pi}B(v,\ell/2);
\]
this covers in particular the fibers over the exterior normal fans at the
convex corners.
\end{itemize}
Then, with
$t_0:=\min\{\rho, \ell/4, (\ell/2)\tan(\beta_{\min}/2), 1/\|\kappa\|_\infty\}$ and
$\beta_v:=2\pi-\theta_v$ (conventions: $\beta_{\min}:=\min_{\theta_v>\pi}(2\pi-\theta_v)$; the angular
term is $+\infty$ if there are no reentrant corners, the last is $+\infty$ if
$\kappa\equiv0$), the expansion \eqref{eq:planar} holds --- equivalently
\eqref{eq:planarGB} --- for all $0<t<t_0$.
\end{proposition}

\begin{remark}\label{rem:planarauto}
The proposition is autonomous: it does not use Theorem~\ref{thm:main}. If $S$
admits a conic certificate of radius $\rho$ with $\Sigma=\emptyset$ --- no
reentrant corners --- then (P1) is vacuous for any $\ell>0$ and (P2) is
precisely condition (C2), with the same $\rho$. If $S$
admits a conic certificate of radius $\rho$ with $\Sigma=$ the nonempty set
of reentrant corners and core radii $r_v$, then (P1)--(P2) hold automatically
with $\ell:=2\min_v r_v,$ $\rho_0:=\min\{\rho, (\ell/2)\min_v\tan (\beta_v/2)\}>0$
in place of $\rho$. Indeed, $\ell\le3r_v$, so the conic equality of (C1)
on $\bar B(v,3r_v)$ restricts to the wedge equality required by (P1), with
$C_v$ the cone of the certificate, and the $\ell/2$-neighbourhoods are
disjoint by the separation of the certificate; positivity of $\rho_0$
follows from Theorem~\ref{thm:link}, which gives
$\tan(\beta_v/2)\ge\rho/r_v$. For (P2), take $(z,u)\in N(S)$ with
$z\notin\bigcup_v B(v,\ell/2)$: either $z$ lies outside every core, where
(C2) gives $\delta_S\ge\rho\ge\rho_0$ directly, or $z$ lies on a straight
wall at distance $a\in[\tfrac\ell2,r_v]$ from its corner $v$, where
Corollary~\ref{cor:trunc} and the wedge fiber \eqref{eq:fiberwedge} give
$\delta_S\ge\bigl(a\tan\tfrac{\beta_v}2\bigr)\wedge\rho\ge\rho_0$.
\end{remark}

\begin{proposition}\label{prop:p2auto}
Let $S=\overline{\operatorname{int}S}$ be compact with boundary a finite
union of piecewise-$C^2$ Jordan curves, and suppose (P1) holds for some
$\ell>0$. Then there is $\rho>0$ such that (P2) holds with this $\ell$.
\end{proposition}

The proof --- a compactness argument whose only geometric inputs are the
$C^2$ bound along each arc and the transversality of the convex corners
--- is in Appendix~\ref{app:pfauto}. Combining it with
Proposition~\ref{prop:planar} proves Theorem~\ref{thm:planarfree}.

\begin{corollary}\label{cor:pwconreach}
Every set satisfying the assumptions of Theorem~\ref{thm:planarfree} has
positive conic reach.
\end{corollary}

\begin{proof}
Let $\ell$ be as in (P1) and let $\rho_0$ be the fiber bound of
Proposition~\ref{prop:p2auto}. Around each reentrant corner $v$ choose
$r_v>0$ with $3r_v\le\ell$ and $|v-v'|>3(r_v+r_{v'})$ for $v\ne v'$; (C1)
holds because $S$ coincides with the wedge on
$\bar B(v,3r_v)\subseteq\bar B(v,\ell)$. For (C2), take $(z,u)\in N(S)$
with $z$ outside every core $B(v,r_v)$. If
$z\notin\bigcup_vB(v,\ell/2)$, then $\delta_S(z,u)\ge\rho_0$. Otherwise
$z$ lies on a wall of some corner $v$, at distance $a\in[r_v,\ell/2]$ from
it. For $y=z+su$ with $0<s<\ell/4$, every foot $q$ of $y$ in $S$
satisfies $|q-v|\le|q-y|+|y-v|\le s+(a+s)<\ell$, so it lies where $S$ and
$v+C_v$ coincide; conversely, every foot of $y$ in $v+C_v$ satisfies the
same bound --- as $\dist(y,v+C_v)\le|y-z|=s$ --- and hence lies in $S$.
Thus $\dist(y,S)=\dist(y,v+C_v)$ with the same feet for all $0<s<\ell/4$,
so $\delta_S(z,u)\wedge\tfrac\ell4=\delta_{v+C_v}(z,u)\wedge\tfrac\ell4$;
since the wedge fiber at $z$ has length $a\tan(\beta_v/2)$, this gives
$\delta_S(z,u)\ge\min\bigl\{a\tan\tfrac{\beta_v}2,\tfrac\ell4\bigr\}$.
Thus (C2) holds with
$\rho:=\min\bigl\{\min_vr_v, \min_vr_v\tan\tfrac{\beta_v}2,\
\tfrac\ell4, \rho_0\bigr\}>0$ (with $\Sigma=\emptyset$ and $\rho:=\rho_0$
when there are no reentrant corners), so $\conreach(S)\ge\rho>0$.
\end{proof}

\begin{remark}\label{rem:trap}
It is tempting to compute the reentrant correction as ``the area counted twice'', i.e.\
as the area of the overlap of the two half-tubes, which gives
$-t^2\bigl(\cot\tfrac\beta2-\cot\beta\bigr)$. This is wrong in general (it
happens to coincide with the correct value only at $\beta=\pi/2$, i.e.\
$\theta_v=3\pi/2$: the two expressions are $1$ and $1$ there). The correct bookkeeping
is by fibers, \eqref{eq:fiberwedge}: what is lost is not an overlap of two
rectangles but the truncation of the normal fibers at the bisector. The
notched square of Example~\ref{ex:notch}, where $\beta\ne\pi/2$,
distinguishes the two formulas unambiguously.
\end{remark}

\begin{example}\label{ex:notch}
$S=$ the square $[0,6]^2$ with the triangle of vertices $(2,0)$, $(3,\sqrt3)$, $(4,0)$
removed, i.e. the polygon with vertices
$(0,0)$, $(2,0)$, $(3,\sqrt3)$, $(4,0)$, $(6,0)$, $(6,6)$, $(0,6)$. Here $\Per=26$, the apex
$(3,\sqrt3)$ is reentrant with $\theta=5\pi/3$, the two notch corners are convex with
$\theta=2\pi/3$, and there are four right angles. Formula \eqref{eq:planarGB} gives
$c_2=\pi+\psi\bigl( {5\pi}/{3}\bigr)
=\pi+\cot({5\pi}/{6})+{\pi}/{3}
={4\pi}/{3}-\sqrt3 .$ 
The ``overlap'' recipe of Remark~\ref{rem:trap} would instead give
${4\pi}/{3}-{2\sqrt3}/{3}$.
\end{example}

\begin{example}\label{ex:pacman}
$S=\{r e^{i\phi}: 0\le r\le 1, \alpha\le\phi\le 2\pi\}$ with
$\alpha\in(0,\pi)$: one reentrant corner at the origin with straight
incident radii, so (P1) holds, and Corollary~\ref{cor:pwconreach} gives
$\conreach(S)>0$. Substitution into \eqref{eq:planarGB} gives
$V_S(t)=(2+2\pi-\alpha)t+\bigl[{3\pi}/2- \alpha/2
+\cot(\pi-\alpha/2)\bigr]t^2$ for small $t$.
\end{example}

\begin{example}\label{ex:cross}
$S=\bigcup_{i=1}^m[-e_i,e_i]\subset\R^2$ ($m$ unit vectors with $e_i\neq\pm e_j$
for $i\neq j$):
$\lambda_2(S)=0$ and $V_S(t)=\lambda_2(S_t)$. Here $\Sigma=\{0\}$ and $C_0=S$'s
directions cone (polyhedral), the $2m$ endpoints have positive reach. Theorem
\ref{thm:main} applies and yields $V_S(t)=4mt+c(e_1,\dots,e_m) t^2$ for small
$t$, where the quadratic coefficient collects the wedge corrections between
consecutive branches and the $2m$ endpoint half-disc caps; we do not compute it
here. Thus the class also contains lower-dimensional sets with genuine
crossings.
\end{example}

\begin{remark}\label{rem:unions}
Let $R>0$. If $S_1,S_2$ are compact with $\dist(S_1,S_2)>2R$, then, for
$0<t<R$, $(S_1\cup S_2)_t=(S_1)_t\mathbin{\dot\cup}(S_2)_t$: a point within
distance $t$ of both sets would give $\dist(S_1,S_2)\le2t<2R$. Hence
$V_{S_1\cup S_2}(t)=V_{S_1}(t)+V_{S_2}(t)$ and
$\polreach(S_1\cup S_2)\ge\min\{\polreach(S_1),\polreach(S_2),R\}$.
\end{remark}

\section{Three failure mechanisms}\label{sec:sharp}

Each example below is a compact planar set with $\polreach=0$, failing conic reach
through one identifiable mechanism. Together they show that the following natural weakenings of exact conicity
do not suffice in general --- $C^1$ tangency, or ``two $C^\infty$ arcs
meeting transversally'' --- and that the metric non-degeneracy carried by
(C2) cannot be dropped. Whether conic reach itself is necessary for local
polynomiality at an isolated failure of positive reach remains open; the
one-dimensional converse of Section~\ref{sec:converse} is the partial
answer we have.

\subsection{A cusp: the exponent $3/2$}\label{subsec:cusp}

Let $S:=\bigl([0,1]\times[-1,0]\bigr) \cup \bigl\{(x,y): 0\le x\le1, x^2\le y\le 1\bigr\},$
the union of two blocks separated by the cuspidal gap
$H=\{0<x<1, 0<y<x^2\}$: the two boundary pieces $y=0$ and $y=x^2$ are $C^\infty$ and meet
tangentially at the origin --- first-order contact: common value and first
derivative, different curvatures --- so the exterior opening degenerates to zero.

For $t>0$, consider first the vertical model, in which distances to the two
walls are measured vertically: the column $\{x=a\}$ contributes length
$\min\{a^2,2t\}$, and
\[
\int_0^1\min\{a^2,2t\} da=2t-\frac{4\sqrt2}{3} t^{3/2},\qquad 0<t\le\tfrac12 .
\]

\begin{proposition}\label{prop:cusp}
There are a polynomial $q\in\R_2[t]$ and $t_1>0$ such that
$V_S(t)=q(t)- {4\sqrt2}/{3}t^{3/2}+O(t^2),$ $0<t\le t_1.$ In particular $\polreach(S)=0$, and the coefficient of the fractional term is the
Euclidean constant, not merely that of the vertical model.
\end{proposition}

Proved in Appendix~\ref{app:pfrest}.

\begin{remark}\label{rem:cuspflag}
It is tempting to run the argument with only Lemma~\ref{lem:cuspcompare}, i.e.\
with a two-sided comparison $C_1(a^2-y)\le\dist\le C_2(a^2-y)$ with uniform
constants. This cannot work: it traps $F(t)$ in a corridor of linear width
in $t$, and a corridor of width $\Theta(t)$ can hide a $t^{3/2}$ term entirely.
The proof in Appendix~\ref{app:pfrest} needs the two extra inputs that
Lemma~\ref{lem:cuspdepth}
provides: the comparison factor tends to $1$ at the tip (which protects the
constant ${4\sqrt2}/3$ coming from the critical region $a\asymp\sqrt t$),
and away from the tip the covered height is $t\sqrt{1+4a^2}$ up to $O(t^2)$
(which makes the slant of the wall an exactly linear, hence harmless,
contribution).
\end{remark}

\subsection{A curved reentrant corner: two discs}\label{subsec:curved}

This example retains transverse boundary arcs and isolates the effect of
curvature at a reentrant corner: the set is a union of two discs, so its
boundary is $C^\infty$ except at two points, where two arcs meet
transversally. Only exact conicity fails.

Let $R>h>0$ and $S:=\bar B\bigl((-h,0),R\bigr)\cup \bar B\bigl((h,0),R\bigr)\subset\R^2 .$
Because $\dist(\cdot,S)=\min_i\dist(\cdot,\bar B_i)$, the parallel set is
$S_t=\bar B((-h,0),R+t)\cup\bar B((h,0),R+t)$: a union of two discs of radius
$r=R+t$ whose centers are at distance $2h$. By the classical lens formula,
\[
\lambda_2(S_t)=2\pi r^2-\Bigl[2r^2\arccos\frac{h}{r}-2h\sqrt{r^2-h^2}\Bigr],
\qquad r=R+t,
\]
so that, with $A(r):=2\pi r^2-2r^2\arccos(h/r)+2h\sqrt{r^2-h^2}$,
\begin{equation}\label{eq:twodiscs}
{ V_S(t)=A(R+t)-A(R) }
\end{equation}
exactly, for all $t>0$.

\begin{proposition}\label{prop:twodiscs}
$V_S$ in \eqref{eq:twodiscs} is real-analytic on a neighbourhood of $0$, but it is not
a polynomial on any interval $(0,\varepsilon)$. Hence $\polreach(S)=0$, although $\partial S$ is
$C^\infty$ except at two points, at which two circular arcs meet at a strictly positive
(reentrant) angle.
\end{proposition}

Proved in Appendix~\ref{app:pfrest}.

\begin{remark}
The mechanism is visible in the fiber picture. Let $S$ contain the half-plane $\{y\le0\}$ near the origin
together with a disc $\bar B((x_0,y_0),R)$ passing through the origin
($x_0^2+y_0^2=R^2$, $x_0<0$), so that $\partial S$ has a reentrant corner at $0$ with
one straight side and one circular side. The equidistant set between the wall and the
disc is $\{(a,\delta): \delta=\dist((a,\delta),\bar B)\}$, i.e.
\begin{equation}\label{eq:deltaquad}
\delta(a)=\frac{(a-x_0)^2+y_0^2-R^2}{2(R+y_0)}
=\gamma a^2+m a ,
\end{equation}
with $\gamma=\frac{1}{2(R+y_0)}>0$ and $m=\frac{|x_0|}{R+y_0}>0$.
The fiber length over the wall point $(a,0)$ is thus a quadratic function of
$a$, not the linear function \eqref{eq:fiberwedge} of the exactly conical
case, and Proposition~\ref{prop:converse} applies: since $\delta$ is not
linear near $0$, the wall contribution
$\int_0^{r_0}\min\{t,\delta(a)\} da$ is not a polynomial of degree at
most two on any interval $(0,\varepsilon)$ --- explicitly, its inverse cut
function is $\delta^{-1}(t)=\bigl(\sqrt{m^2+4\gamma t}-m\bigr)/(2\gamma)$,
which is not affine. A nonzero curvature of either side destroys the exact homogeneity
$\delta(\lambda a)=\lambda\delta(a)$ of the fiber-cut function and --- in this
example --- with it polynomiality. This strongly suggests that one-homogeneity of
the fiber-cut function, rather than straightness of the incident arcs, is the
relevant invariant; it does not by itself exclude special configurations or
cancellations between several singular points, cf. Remark~\ref{rem:strict}.
\end{remark}

\subsection{A fractal link: the Cantor fan}\label{subsec:cantor}

Let $K\subset[0,\tfrac\pi2]$ be the ternary Cantor set, $D=\log2/\log3$, and
\[
S:=\{r e^{i\phi}: 0\le r\le 1, \phi\in K\}\subset\R^2 ,
\]
the cone over $K$ truncated at radius $1$. Then $S$ is exactly a cone near the
origin, so (C1) holds, but the link $K$ has infinitely many components with
arbitrarily small gaps --- the configuration that Theorem~\ref{thm:link}
forbids under (C2).

\begin{proposition}\label{prop:cantor}
The truncated Cantor fan has $\conreach(S)=0$: no conic certificate of
any radius exists. Moreover there are $0<c_1\le c_2$ and $t_0>0$ with
$c_1 t^{1-D}  \le  V_S(t)  \le  c_2 t^{1-D},$ $0<t\le t_0 ;$
in particular $V_S(t)/t\to\infty$, the outer Minkowski content of $S$ is
infinite, and $\polreach(S)=0$.
\end{proposition}

Proved in Appendix~\ref{app:pfrest}; the exponent agrees with the
Minkowski dimension $1+D$ of the fan.

\appendix

\section{Technical lemmas}\label{app:tech}

\subsection{The fiber length and the medial axis}\label{app:background}

\begin{lemma}
\label{lem:deltameas}
For every closed $S$, $x\in S$ and unit $u$,
$\delta_S(x,u)=\sup\{s>0:  \dist(x+su,S)=s\}$ $(\sup\emptyset:=0),$ and for every $0<s<\delta_S(x,u)$ the point $x+su$ has the unique foot $x$.
Moreover $\{(x,u):\delta_S(x,u)\ge s\}=\{(x,u):\dist(x+su,S)\ge s\}$ is closed for
each $s>0$, so $\delta_S$ is upper semicontinuous, hence Borel, on
$S\times\Sph^{1}$.
\end{lemma}

\begin{proof}
The segment properties behind the first claim are classical,
cf.\ \cite[Thm.~4.8]{federer}; we include the short argument, in the form we
use. Write $\delta^*(x,u)$ for the supremum on the right. If $\dist(x+s_0u,S)=s_0$ and
$0<s<s_0$, then for every $z\in S$,
$|x+su-z|\ge|x+s_0u-z|-(s_0-s)\ge s$, so $\dist(x+su,S)=s$ (the value $s$ is
attained at $z=x$); if some $z\ne x$ also attained it, the two inequalities would
be equalities, forcing $z=x+\alpha u$ with $|s-\alpha|=s$ and $|s_0-\alpha|=s_0$,
i.e. $\alpha=0$, $z=x$. Thus $\{s:\dist(x+su,S)=s\}$ is an initial interval, on whose interior the
foot is unique and equal to $x$; and if $\delta^*<\infty$, taking
$s_n\uparrow\delta^*$ and using continuity of the distance function gives
$\dist(x+\delta^*u,S)=\lim_n s_n=\delta^*$, so the interval is exactly
$(0,\delta^*]$ (and it is $(0,\infty)$ when $\delta^*=\infty$). Consequently $\xi_S(x+su)=x$ for all $s<\delta^*$, giving
$\delta_S\ge\delta^*$; and $\xi_S(x+su)=x$ implies $\dist(x+su,S)=s$, giving
$\delta_S\le\delta^*$. The set equality in the last claim follows since
$\dist(x+su,S)\le s$ always (as $x\in S$), together with the interval structure;
closedness is continuity of $\dist$.
\end{proof}

\begin{lemma}\label{lem:federer}
(i) The set $\R^2\setminus\Unp(S)$ of points with non-unique projection is
Lebesgue-null. (ii) If $\reach(S,x)\ge\rho$ then $\delta_S(x,u)\ge\rho$ for
every $u\in\Nor(S,x)$.
\end{lemma}

\begin{proof}
(i) follows from a.e. differentiability of the locally semiconcave
function $\dist(\cdot,S)^2$ \cite{CS}, whose points of differentiability
have a unique foot. (ii) we quote as a standard fact, contained in \cite[Thm.~4.8]{federer}:
within the reach the metric projection is well defined and every normal
segment projects back to its foot, i.e. the local reach bounds the proximal
fiber lengths from below (equivalently,
$\reach(S)=\inf\{\delta_S(x,u):(x,u)\in N(S)\}$; for a modern account of
positive reach see also \cite{RZbook}); it is used here only for
Remark~\ref{rem:reachincl}.
\end{proof}

\subsection{Localization to the cone}\label{app:loclem}

\begin{lemma}\label{lem:local}
Let $\sigma\in\Sigma$ and $y\in\bar B(\sigma,2r_\sigma)$ with
$0<\dist(y,S)<\rho$. Then
$\dist(y,S)=\dist(y,\sigma+C_\sigma),$ $\feet_S(y)=\feet_{\sigma+C_\sigma}(y),$
and all feet lie in $B(\sigma,3r_\sigma)$.
\end{lemma}

\begin{proof}
Any foot of $y$ in $S$ lies within
$|y-\sigma|+\dist(y,S)<2r_\sigma+\rho\le3r_\sigma$ of $\sigma$, hence in
$S\cap\bar B(\sigma,3r_\sigma)=(\sigma+C_\sigma)\cap\bar B(\sigma,3r_\sigma)$;
so $\dist(y,S)=\dist\bigl(y,S\cap\bar B(\sigma,3r_\sigma)\bigr)\ge
\dist(y,\sigma+C_\sigma)$. Conversely, a nearest point of the cone to $y$
lies within $|y-\sigma|+\dist(y,\sigma+C_\sigma)\le
2r_\sigma+\dist(y,S)<3r_\sigma$ of $\sigma$, hence belongs to $S$. The two
distances and feet sets therefore coincide.
\end{proof}

The localization is exact, so it transfers truncated fiber lengths in both
directions.

\begin{corollary}\label{cor:trunc}
Let $S$ admit a conic certificate of radius $\rho$, let $\sigma\in\Sigma$,
and let $z\in\partial S$ with $|z-\sigma|\le 2r_\sigma-\rho$ (in
particular, any $z$ with $|z-\sigma|=r_\sigma$ qualifies, as
$\rho\le r_\sigma$). Then, for every unit vector $u$,
$\delta_S(z,u)\wedge \rho = \delta_{\sigma+C_\sigma}(z,u)\wedge \rho .$
\end{corollary}

\begin{proof}
Write $C:=\sigma+C_\sigma$; by (C1), $S$ and $C$ coincide on
$\bar B(\sigma,3r_\sigma)$, and $|z-\sigma|<3r_\sigma$, so
$z\in\partial C$ as well (equal sets in an open ball have equal boundaries
there). Suppose
$\delta_C(z,u)\wedge\rho<\delta_S(z,u)\wedge\rho$ and pick $s$ strictly
between the two values; then $s<\rho$ and $s<\delta_S(z,u)$, so
$\dist(z+su,S)=s$ with unique foot $z$ (Lemma~\ref{lem:deltameas}), and
$|z+su-\sigma|\le(2r_\sigma-\rho)+s<2r_\sigma$; Lemma~\ref{lem:local}
gives $\feet_C(z+su)=\feet_S(z+su)=\{z\}$ and $\dist(z+su,C)=s$, whence
$s\le\delta_C(z,u)$ by Lemma~\ref{lem:deltameas} --- contradicting
$s>\delta_C(z,u)\wedge\rho$. In the symmetric case
$\delta_S(z,u)\wedge\rho<\delta_C(z,u)\wedge\rho$, pick $s$ strictly
between the two values: then $\dist(z+su,C)=s$ with unique foot $z$;
moreover $z+su\in B(\sigma,2r_\sigma)\subset\bar B(\sigma,3r_\sigma)$,
where $S$ and $C$ coincide, and $\dist(z+su,C)=s>0$ gives $z+su\notin S$,
while $z\in S$ gives $0<\dist(z+su,S)\le s<\rho$;
Lemma~\ref{lem:local} applies again and yields $\dist(z+su,S)=s$ with foot
$z$, so $s\le\delta_S(z,u)$ --- a contradiction.
\end{proof}

\subsection{The side fibers of a planar cone}\label{app:sidelem}

\begin{lemma}\label{lem:sidefiber}
Let $C\subseteq\R^2$ be a cone whose link $K=C\cap\Sph^1$ is closed and
$\neq\Sph^1$, let $G$ be a connected component of $\Sph^1\setminus K$ of
angular width $\theta\in(0,2\pi]$, let $\omega_1$ be an endpoint of $G$, and
let $u\perp\omega_1$ point into $G$. Then $(\omega_1,u)\in N(C)$ and
\[
\tilde\delta(\omega_1,u) = 
\begin{cases}
\tan(\theta/2), & \theta<\pi,\\[2pt]
\infty, & \theta\ge\pi .
\end{cases}
\]
Moreover, the proximal unit normals of $C$ at nonzero points are exactly
indexed by the incidences $(G,\omega)$ of a gap $G$ and one of its two
endpoints $\omega$: for each incidence, the unique proximal direction at
the points $r\omega$, $r>0$, is the perpendicular to $\omega$ pointing into
$G$, with normalized fiber length as displayed; there are no other proximal
normals at nonzero points of $C$. (A one-point component of $K$ carries two
incidences, one for each side of its ray.)
\end{lemma}

\begin{proof}
Rotate so that $\omega_1=(1,0)$, $u=(0,1)$, and $G=\{(\cos\phi,\sin\phi):
0<\phi<\theta\}$; then $K\subseteq\{(\cos\phi,\sin\phi):\theta\le\phi\le
2\pi\}$. Fix $s>0$, put $p:=\omega_1+su=(1,s)$ and $\alpha:=\arctan s\in
(0,\tfrac\pi2)$, so $|p|=\sqrt{1+s^2}$ and $s=|p|\sin\alpha$. For a ray
$\R_{\ge0}\omega'$ with $\omega'$ at angle $\phi$, write $\psi:=\phi-\alpha$;
then $\langle p,\omega'\rangle=|p|\cos\psi$ and
\[
\dist\bigl(p,\R_{\ge0}\omega'\bigr)=
\begin{cases}
|p| |\sin\psi|, & \cos\psi>0,\\
|p|, & \cos\psi\le0 .
\end{cases}
\]

Case $\theta\ge\pi$. Every $\omega'\in K$ has angle $\phi\in[\theta,2\pi]
\subseteq[\pi,2\pi]$, hence $\langle w-\omega_1,u\rangle=r\sin\phi\le0$ for
every $w=r\omega'\in C$: the set $C$ lies in the closed half-plane
$\{w:\langle w-\omega_1,u\rangle\le0\}$, so
$|p-w|^2=s^2-2s\langle u,w-\omega_1\rangle+|w-\omega_1|^2\ge
s^2+|w-\omega_1|^2$ for all $w\in C$, with strict inequality unless
$w=\omega_1$. Thus $\dist(p,C)=s$ with unique foot $\omega_1$ for every
$s>0$: $\tilde\delta(\omega_1,u)=\infty$.

Case $\theta<\pi$, lower bound. Let $0<s<\tan(\theta/2)$, i.e.
$\alpha<\theta/2$. For $\omega'\in K$ at angle $\phi\in[\theta,2\pi)$,
$\psi=\phi-\alpha\in[\theta-\alpha,2\pi-\alpha)$. If $\cos\psi\le0$ the
distance to that ray is $|p|>s$. If $\cos\psi>0$ then either
$\psi\in[\theta-\alpha,\tfrac\pi2)$, and
$|\sin\psi|=\sin\psi\ge\sin(\theta-\alpha)>\sin\alpha$ because
$\alpha<\theta-\alpha\le\tfrac\pi2$; or $\psi\in(\tfrac{3\pi}2,2\pi-\alpha)$,
and $|\sin\psi|=\sin(2\pi-\psi)>\sin\alpha$ because
$2\pi-\psi\in(\alpha,\tfrac\pi2)$. In every case
$\dist(p,\R_{\ge0}\omega')>|p|\sin\alpha=s$, while the ray through
$\omega_1$ itself gives $\dist(p,\R_{\ge0}\omega_1)=s$, attained only at
$\omega_1$. Hence $\dist(p,C)=s$ with unique foot $\omega_1$, and
$\tilde\delta(\omega_1,u)\ge\tan(\theta/2)$.

Case $\theta<\pi$, upper bound. Let $s>\tan(\theta/2)$, i.e.
$\alpha>\theta/2$, and let $\omega_2$ be the other endpoint of $G$, at angle
$\theta$. Then $\psi=\theta-\alpha\in(-\alpha,\theta/2)$ satisfies
$|\psi|<\alpha<\tfrac\pi2$, so $\cos\psi>0$ and
$\dist(p,\R_{\ge0}\omega_2)=|p| |\sin(\theta-\alpha)|<|p|\sin\alpha=s$.
Hence $\dist(p,C)<s$ for every $s>\tan(\theta/2)$, and
$\tilde\delta(\omega_1,u)\le\tan(\theta/2)$.

Finally we verify the classification, for an arbitrary closed $K$. Let
$\omega\in K$ and let $u=a\omega+b\omega^{\perp}$ be a unit vector,
$a^2+b^2=1$. If $a\ne0$, then for small $s>0$ the point
$(1+sa)\omega\in C$ satisfies $|\omega+su-(1+sa)\omega|=s|b|<s$, so
$\omega$ is not a foot of $\omega+su$ and $u$ is not proximal. If $a=0$
then $u=\pm\omega^{\perp}$ points to one of the two sides of $\omega$, and
either that side carries a gap adjacent to $\omega$, or $K$ accumulates at
$\omega$ from that side. In the first case $u$ is the proximal normal of
the corresponding incidence, computed above. In the second case $u$ is not
proximal: rotate so that $\omega=(1,0)$ and $u=(0,1)$, fix $s>0$, and put
$p:=\omega+su=(1,s)$, whose polar angle is $\alpha=\arctan s$ and whose
norm satisfies $|p|\sin\alpha=s$. By accumulation there is
$\omega_\phi\in K$ at angle $\phi$ with $0<\phi<2\alpha$, and the distance
from $p$ to the ray $\R_{\ge0} \omega_\phi\subseteq C$ is
$|p|\sin|\alpha-\phi|<|p|\sin\alpha=s=|p-\omega|$: so $\omega$ is not a
foot of $p$, for any $s>0$. (This covers in particular points in the
relative interior of a nondegenerate arc, where both sides accumulate;
consistently, such points $r\omega$ are interior to $C$ and
$\delta_C\equiv0$ there.)
\end{proof}

\subsection{Two lemmas for the cusp}\label{app:cusplem}

The following two lemmas concern the cuspidal set of Section~\ref{subsec:cusp}.
Write $G:=\{(x,x^2):0\le x\le1\}$ for the parabolic wall.

\begin{lemma}\label{lem:cuspcompare}
For $0<a\le1$ and $0\le y\le a^2$,
\[
\frac{a^2-y}{\sqrt{1+16a^2}} \le \dist\bigl((a,y),G\bigr) \le a^2-y .
\]
\end{lemma}

\begin{proof}
The upper bound is the vertical competitor $(a,a^2)$. For the lower bound, any
minimizer $(x^\ast,x^{\ast2})$ satisfies $|x^\ast-a|\le\dist\le a^2\le a$, so
$x^\ast\in[0,2a]$, where $x\mapsto x^2$ is $4a$-Lipschitz; hence
\[
a^2-y\le|a^2-x^{\ast2}|+|x^{\ast2}-y|
\le 4a |a-x^\ast|+|x^{\ast2}-y|
\le\sqrt{1+16a^2} \bigl|(a,y)-(x^\ast,x^{\ast2})\bigr|
\]
by the Cauchy--Schwarz inequality applied to the vector $(4a,1)$.
\end{proof}

\begin{lemma}\label{lem:cuspdepth}
Let $0<t\le\tfrac1{100}$, $3\sqrt t\le a\le1-t$, and let $p=(a,y)$ lie in the
gap with $h:=\dist(p,G)\le t$. Then
$|(a^2-y) - h\sqrt{1+4a^2}| \le  5 h^2 .$
\end{lemma}

\begin{proof}
A minimizer $x^\ast$ of $\varphi(x)=(a-x)^2+(y-x^2)^2$ on $[0,1]$ satisfies
$|x^\ast-a|\le h\le t\le a/3$, so $x^\ast\ge2a/3>0$; and $x^\ast<1$ because
$\varphi'(1)=2(1-a)+4(1-y)>0$. Hence $x^\ast$ is interior and
$\varphi'(x^\ast)=0$, i.e. $p-q\perp$ the tangent at $q=(x^\ast,x^{\ast2})$. Since
$p$ lies below the graph, $p=q+h (2x^\ast,-1)/\sqrt{Q}$ with $Q:=1+4x^{\ast2}$.
Substituting,
\[
a=x^\ast+\frac{2x^\ast h}{\sqrt Q},\qquad y=x^{\ast2}-\frac{h}{\sqrt Q},
\qquad\text{whence}\qquad
a^2-y=h\sqrt Q+\frac{4x^{\ast2}h^2}{Q}.
\]
Finally $|\sqrt Q-\sqrt{1+4a^2} |\le 2|x^\ast-a|\le2h$ (the map
$x\mapsto\sqrt{1+4x^2}$ is $2$-Lipschitz), and $4x^{\ast2}/Q\le1$, giving the
claim with constant $1+2\le5$ margin to spare.
\end{proof}

\section{Proofs of the main results}\label{app:proofs}

\subsection{Proof of Theorem~\ref{thm:link}}\label{app:pflink}

\begin{proof}
If $K_\sigma=\Sph^1$ then $C_\sigma=\R^2$ and, by (C1),
$\bar B(\sigma,3r_\sigma)\subseteq S$, contradicting
$\sigma\in\partial S$. If $K_\sigma=\emptyset$ the statement holds
trivially under our convention: there are no gaps and zero arcs. Assume
from now on $\emptyset\ne K_\sigma\ne\Sph^1$, and let $G$ be a gap of width
$\theta<\pi$ (gaps of width $\ge\pi$ already exceed $\theta_\sigma$, since
$\rho\le r_\sigma$ gives $\theta_\sigma\le\pi/2$), let $\omega_1$ be an endpoint of $G$ and
$u\perp\omega_1$ the direction into $G$. Put
$z:=\sigma+r_\sigma \omega_1$, a boundary point of the cone in the
interior of the conicity ball, hence $z\in\partial S$ by (C1). The point
$z$ lies outside every core: $|z-\sigma|=r_\sigma$, so
$z\notin B(\sigma,r_\sigma)$, and for every other singular point,
$|z-\sigma'|\ge|\sigma-\sigma'|-r_\sigma>3(r_\sigma+r_{\sigma'})-r_\sigma
>r_{\sigma'}$. Thus (C2) applies at $z$. By Lemma~\ref{lem:sidefiber} and
homogeneity, $\delta_{\sigma+C_\sigma}(z,u)=r_\sigma\tan(\theta/2)$, which
is positive, so $(z,u)\in N(\sigma+C_\sigma)$; by
Corollary~\ref{cor:trunc} the pair $(z,u)$ is also proximal for $S$, and
(C2) forces $\delta_S(z,u)\wedge\rho=\rho$, whence
$\rho = \delta_S(z,u)\wedge \rho = 
\bigl(r_\sigma\tan\tfrac\theta2\bigr)\wedge \rho ,$
i.e. $r_\sigma\tan(\theta/2)\ge\rho$ and $\theta\ge\theta_\sigma$. The
complement of $K_\sigma$ is a disjoint union of open arcs, each of width
$\ge\theta_\sigma$, so there are at most $\lfloor2\pi/\theta_\sigma\rfloor$
of them; $K_\sigma$ is the complement of their union, a finite union of
pairwise disjoint closed arcs.
\end{proof}

\subsection{Proof of Theorem~\ref{thm:main}}\label{app:pfmain}

The regular zone and the conical zones are treated separately and then
summed; the regular contribution is a direct consequence of the local
Steiner formula.

\begin{proposition}\label{prop:regular}
Let $S$ admit a conic certificate of radius $\rho$ and let
$\eta=N(S)\cap(\Gamma_{\mathrm{reg}}\times\Sph^1)$ as in
Theorem~\ref{thm:main}. Then $\eta$ is $r$-bounded, both $\Theta_j(\eta)$
are finite, and
\[
\vol\bigl(W_t(\eta)\bigr)=\Theta_1(\eta) t+\tfrac12 \Theta_0(\eta) t^2,
\qquad 0<t<\rho .
\]
\end{proposition}

\begin{proof}
The feet of $\eta$ lie in the compact set $\partial S$ and, by (C2) ---
which applies exactly on
$\Gamma_{\mathrm{reg}}=\partial S\setminus\bigcup_\sigma
B(\sigma,r_\sigma)$ --- we have $\delta_S\ge\rho$ on $\eta$: so $\eta$ is
$r$-bounded, so the $\Theta_j$ are finite signed measures on it
(Theorem~\ref{thm:HLW}). For $0<t<\rho$ we have
$\min\{t,\delta_S\}=t$ identically on $\eta$, and \eqref{eq:HLW} reads as
displayed.
\end{proof}

\begin{proof}[Proof of Proposition~\ref{prop:cone}]
We partition the defining set by the location of the foot. If
$K=\emptyset$ then $C=\{0\}$ and the set is the punctured disc
$\bar B(0,t)\setminus\{0\}$, of area $\pi t^2$, matching $m=0$,
$\gamma_C=\pi$. Assume $K\ne\emptyset$.

Feet at the vertex. For a unit $u$, we claim $\delta_C(0,u)=\infty$ if
$d_{\Sph^1}(u,K)\ge\pi/2$ and $=0$ otherwise. If
$\langle u,\omega\rangle\le0$ for every $\omega\in K$ --- which is exactly
the angular-distance condition --- then for $w=r\omega\in C$,
$|su-w|^2=s^2-2sr\langle u,\omega\rangle+r^2\ge s^2+r^2$, so
$\dist(su,C)=s$ with unique foot $0$ for every $s>0$. If instead
$\langle u,\omega_0\rangle>0$ for some $\omega_0\in K$, then
$|su-r\omega_0|<s$ for small $r>0$, so $\dist(su,C)<s$ for every $s>0$ and
$0$ is never the foot. The vertex-footed part of the tube is therefore the
union of the full segments $\{su:0<s\le t\}$ over the fan
$F:=\{u:d_{\Sph^1}(u,K)\ge\pi/2\}$, of area
$\tfrac12 t^2 \mathcal H^1(F)=\gamma_C t^2$ (polar coordinates).

Feet on the rays. Let $y\in\Unp(C)$ with $0<\dist(y,C)\le t$ and
$\xi_C(y)=r\omega$, $0<r<r_0$, $\omega\in K$ (by the classification of
Lemma~\ref{lem:sidefiber}, the tube points whose foot lies on the circle
$|\xi_C(y)|=r_0$ form a finite union of straight normal segments
$\{r_0\omega+su:0<s\le t\}$, one per incidence, of measure zero, and are
discarded). Since $r$ is an
interior minimum of $\rho\mapsto|y-\rho\omega|^2$ over $\rho>0$, we get
$\langle y-r\omega,\omega\rangle=0$: writing $s:=\dist(y,C)$ and
$u:=(y-r\omega)/s$, the direction $u$ is one of $\pm\omega^\perp$ and is
proximal at $\omega$, so by Lemma~\ref{lem:sidefiber} $\omega$ is an
endpoint of some gap $G_i$ and $u$ points into $G_i$, with
$0<s<r\Delta(\theta_i)$; for $\theta_i<\pi$ this holds up to the cut set
$\{s=r\Delta(\theta_i)\}$ (a line through the origin in the
$(r,s)$-coordinates, of measure zero; on it the foot need not be unique),
while for $\theta_i\ge\pi$ the fiber is unbounded and there is no cut set. Conversely, the lower-bound case of
Lemma~\ref{lem:sidefiber}, scaled by $r$, shows that every point
$r\omega+su$ with $\omega,u$ as above, $0<r<r_0$ and
$0<s<\min\{t, r\Delta(\theta_i)\}$ belongs to the set, with unique foot
$r\omega$. Each side
--- each of the $2m$ pairs (gap, endpoint) --- is thus parametrized by the
map $(r,s)\mapsto r\omega+su$, a Euclidean isometry of the plane onto
itself in the orthonormal frame $(\omega,u)$, so its contribution has area
\[
\int_0^{r_0}\min\bigl\{t, r\Delta(\theta_i)\bigr\} dr
 = r_0 t-\frac{t^2}{2\Delta(\theta_i)},
\qquad t\le r_0\Delta(\theta_i) ,
\]
with the convention $1/\infty:=0$: for $\Delta(\theta_i)=+\infty$ the
integrand is identically $t$ and the integral is $r_0t$, while for finite
$\tilde\delta:=\Delta(\theta_i)$ and $t\le r_0\tilde\delta$,
$\int_0^{t/\tilde\delta}r\tilde\delta dr+\int_{t/\tilde\delta}^{r_0}t dr
=\tfrac{t^2}{2\tilde\delta}+t\bigl(r_0-\tfrac t{\tilde\delta}\bigr)
=r_0t-\tfrac{t^2}{2\tilde\delta}$. Distinct sides parametrize disjoint sets
up to the Lebesgue-null set of points with more than one foot
(Lemma~\ref{lem:federer}(i)), and every point of the defining set is
covered. Summing the vertex fan and the $2m$ sides --- each gap $G_i$
contributes its two endpoints, hence twice
$r_0t-\tfrac12\cot(\theta_i/2)t^2$ when $\theta_i<\pi$ and twice $r_0t$ when
$\theta_i\ge\pi$ --- yields the formula.
\end{proof}

It remains to assemble the two zones.

\begin{proof}
Fix $0<t<\rho$ and recall $W_t=\{y:0<\dist(y,S)\le t\}$, so
$V_S(t)=\vol(W_t)$. Set $\Gamma_{\mathrm{sing}}:=\partial S\cap
\bigcup_{\sigma\in\Sigma}B(\sigma,r_\sigma)$; the separation of the
certificate makes these balls pairwise disjoint, and
$\partial S=\Gamma_{\mathrm{reg}}\sqcup\Gamma_{\mathrm{sing}}$. Up to the
Lebesgue-null set of points with more than one foot
(Lemma~\ref{lem:federer}(i)), $W_t$ is partitioned by the location of the
unique foot:
\[
V_S(t)=\vol\bigl(W_t(\eta)\bigr)
+\sum_{\sigma\in\Sigma}\vol\bigl(W_t(\partial S\cap B(\sigma,r_\sigma))\bigr),
\]
with $\eta$ as in the statement (a foot $z\in\Gamma_{\mathrm{reg}}$ with the
direction $u$ towards $y$ is a pair of $\eta$).

We claim that, for each $\sigma$, the singular term equals
$T_{C_\sigma,r_\sigma}(t)$, computed for the pure cone. Translate $\sigma$
to the origin, keep writing $S$ for $S-\sigma$, and put $C:=C_\sigma$,
$r:=r_\sigma$; by (C1), $S$ and $C$ coincide on $\bar B(0,3r)$. If
$y\in W_t(\partial S\cap B(0,r))$ then $0<\dist(y,S)\le t<\rho$ and
$|y|<r+t<2r$, so Lemma~\ref{lem:local} gives $\dist(y,C)=\dist(y,S)$ and
$\feet_C(y)=\feet_S(y)=\{\xi_S(y)\}$: thus $y\in\Unp(C)$,
$0<\dist(y,C)\le t$ and $\xi_C(y)\in B(0,r)$. Conversely, let
$y\in\Unp(C)$ with $0<\dist(y,C)\le t$ and $\xi_C(y)\in B(0,r)$; then
$\xi_C(y)\in C\cap\bar B(0,3r)=S\cap\bar B(0,3r)$, so
$\dist(y,S)\le\dist(y,C)\le t<\rho$, and $|y|<r+t<2r$; moreover
$y\notin C$ and $S=C$ on $\bar B(0,3r)\ni y$, so $y\notin S$ and
$\dist(y,S)>0$; Lemma~\ref{lem:local} applies and gives
$\dist(y,S)=\dist(y,C)$ and $\feet_S(y)=\feet_C(y)$, whence
$y\in W_t(\partial S\cap B(0,r))$. The two sets coincide; and since the
tube points whose cone foot lies exactly on the circle $|\xi_C(y)|=r$ form
a Lebesgue-null set (they are discarded in the proof of
Proposition~\ref{prop:cone}), the volume over the open foot-ball equals
the volume over the closed one: the singular term is $T_{C,r_\sigma}(t)$.

It remains to apply Proposition~\ref{prop:regular} (valid since $t<\rho$)
and, at each $\sigma$, Proposition~\ref{prop:cone} with $r_0=r_\sigma$: by
Theorem~\ref{thm:link}, every gap width $\theta_i<\pi$ of $K_\sigma$ has
$\tan(\theta_i/2)\ge\rho/r_\sigma$, so
$r_0\Delta_{\min}\ge r_\sigma\cdot\tfrac{\rho}{r_\sigma}=\rho>t$
and the cone formula is in force. Substituting,
\[
V_S(t)=\Theta_1(\eta) t+\tfrac12\Theta_0(\eta) t^2
+\sum_{\sigma\in\Sigma}\Bigl[2m_\sigma r_\sigma t
+\Bigl(\gamma_\sigma-\sum_{\theta_i<\pi}\cot\tfrac{\theta_i}2\Bigr)t^2\Bigr],
\]
which is \eqref{eq:mainformula}. Finally, for every certificate of radius
$\rho$, $V_S$ agrees on $(0,\rho)$ with a polynomial vanishing at $0$ of
degree $\le2$; any two such polynomials agree near $0$, hence coincide, and
by continuity ($V_S(0^+)=0$) the agreement extends to the closed intervals
$[0,R']$, $R'<\rho$, as Definition~\ref{def:polreach} requires. Taking the
supremum over certifiable radii gives $\polreach(S)\ge\conreach(S)$.
\end{proof}

\subsection{Proof of Proposition~\ref{prop:planar}}\label{app:pfplanar}

\begin{proof}
Partition the tube $W_t$ by the location of the foot, as in the proof of
Theorem~\ref{thm:main}; the set
of points with two feet is null.

(a) Arc points away from reentrant corners. If $z$ lies on an arc outside
the $\ell/2$-neighbourhoods of the reentrant corners, its fibers have length
$\ge\rho>t$ by (P2), and the Jacobian $1+s\kappa$ is positive for $s<t<t_0$.
The normal map $(\sigma,s)\mapsto z(\sigma)+s n(\sigma)$ is moreover injective
on this range: if two parameter pairs $(\sigma_i,s_i)$ produced the same point
$y$, then $s_i<t<\rho\le\delta_S(z(\sigma_i),n(\sigma_i))$ by (P2), and
Lemma~\ref{lem:deltameas} makes each $z(\sigma_i)$ the unique foot of $y$;
hence $z(\sigma_1)=z(\sigma_2)$, and with it $n(\sigma_1)=n(\sigma_2)$ (on a
regular arc the outward normal is determined by the foot) and $s_1=s_2$. The
area formula for the normal map, with Jacobian $1+s\kappa(\sigma)$, then shows that this part of the
tube has area $\int(t+\tfrac12\kappa t^2) ds$ over the corresponding arc
portion.

(b) Convex corners. Let $\theta_v<\pi$. Exterior points near $v$ project either
onto one of the two arcs, or onto $v$ itself; the latter set is the circular sector of
radius $t$ spanned by the exterior normal fan at $v$, of opening $\pi-\theta_v$, whose
area is $\tfrac12(\pi-\theta_v)t^2$. Moreover no fiber over the two incident arcs is
cut near $v$: (P2) covers the fibers over the incident arcs and the normal fan
at $v$, so every such fiber has length $\ge\rho>t$, and the argument of (a)
applies up to $v$ with no deficit. Hence a convex corner contributes exactly $\tfrac12(\pi-\theta_v)t^2$ and the
arcs contribute their full $\int(t+\tfrac12\kappa t^2)ds$ up to $v$.

(c) Reentrant corners. (The straight walls inside the $\ell/2$-neighbourhood
of a reentrant corner are accounted for here and only here: their full
fiber integral is computed below as the perimeter term minus the triangular
deficit, so there is no overlap with part (a).) Let $\theta_v>\pi$ and
$\beta:=2\pi-\theta_v\in(0,\pi)$
be the opening of the exterior wedge. By (P1), near $v$ the set is exactly the cone
$v+C_v$, $C_v$ the solid wedge of opening $\theta_v$. Put $v$ at the origin and the
two walls along the rays $\{a e_1: a\ge0\}$ and $\{a(\cos\beta,\sin\beta):a\ge0\}$, the
exterior wedge being $\{0<\arg y<\beta\}$. Exterior points project onto the two walls
(never onto $v$: for $y$ at polar radius $r_y$ inside the wedge,
$\dist(y,\text{wall})\le r_y\sin(\beta/2)<r_y=|y-v|$), the wall being selected by the
bisector $\{\arg y=\beta/2\}$. The fiber over the wall point $(a,0)$ is the vertical
segment stopped by the bisector, so its length is
\begin{equation}\label{eq:fiberwedge}
\delta(a)=a \tan\frac{\beta}{2}\qquad(\text{exactly linear, by homogeneity}).
\end{equation}
By (P1), $S$ coincides with the wedge $v+C_v$ on $\bar B(v,\ell)$, and
this alone localizes the fibers, since $t<t_0\le\ell/4$: for a wall point
$z$ with $|z-v|\le\ell/2$ and $y=z+su$ with $0<s\le t$, every foot $w$ of
$y$ in $S$ satisfies
$|w-v|\le\dist(y,S)+|y-v|\le s+(\ell/2+s)<\ell$, hence is a point of the
wedge; conversely every wedge foot of $y$ lies within
$|y-v|+\dist(y,v+C_v)\le\ell/2+2s<\ell$ of $v$, hence is a point of $S$.
The distances and feet therefore coincide, exactly as in
Lemma~\ref{lem:local}: these wall fibers are those of the pure wedge, and
the computation below is exact.
Comparing with the ``uncut'' value $t$ we obtain --- for
$t\le(\ell/2)\tan(\beta/2)$, which keeps $a^*=t\cot(\beta/2)$ inside the treated
wall of length $\ell/2$ --- a deficit per wall of
\[
\int_0^{a^*}\bigl(t-a\tan\tfrac\beta2\bigr) da=\frac{t^2}{2}\cot\frac\beta2,
\qquad a^*=t\cot\frac\beta2 ,
\]
hence a total deficit $t^2\cot(\beta/2)$ for the two walls. Since
$\cot(\beta/2)=\cot(\pi-\theta_v/2)=-\cot(\theta_v/2)$, the reentrant corner
contributes $+\cot(\theta_v/2) t^2$ (a negative amount) on top of the perimeter term,
and no vertex sector.

Adding (a), (b), (c) over the finitely many arcs and corners, the $t$-coefficient is
$\mathcal H^1(\partial S)=\Per(S)$ and the $t^2$-coefficient is \eqref{eq:planar}. For
\eqref{eq:planarGB} use Gauss--Bonnet in the form
$\int_{\partial S_{\mathrm{reg}}}\kappa ds+\sum_v(\pi-\theta_v)=2\pi\chi(S)$: substituting
$\tfrac12\int\kappa=\pi\chi-\tfrac12\sum_v(\pi-\theta_v)$ into \eqref{eq:planar}, the
convex terms cancel and each reentrant corner leaves
$\cot(\theta_v/2)-\tfrac12(\pi-\theta_v)=\psi(\theta_v)$.
\end{proof}

\subsection{Proof of Proposition~\ref{prop:p2auto}}\label{app:pfauto}

\begin{proof}
Set $K:=\partial S\setminus\bigcup_{v: \theta_v>\pi}B(v,\ell/2)$, a
compact subset of $\partial S$. Suppose the conclusion fails: there are
$(z_n,u_n)\in N(S)$ with $z_n\in K$ and
$\delta_n:=\delta_S(z_n,u_n)\to0$. Passing to a subsequence, $z_n\to z\in
K$. Two observations are used throughout. First, for every $s>\delta_n$
there is a competitor $q\in\partial S$ with $|z_n+su_n-q|<s$, and any
such $q$ satisfies $|z_n-q|\le|z_n+su_n-q|+s<2s$. Indeed, by
Lemma~\ref{lem:deltameas}, $\dist(z_n+su_n,S)<s$. If $z_n+su_n\notin S$,
take for $q$ a foot of $z_n+su_n$ in $S$: feet of exterior points lie on
$\partial S$. If $z_n+su_n\in S$, let
$r_*:=\inf\{r>0: z_n+ru_n\in S\}$; since $\dist(z_n+ru_n,S)=r>0$ for
$0<r<\delta_n$, we have $\delta_n\le r_*\le s$, and
$q:=z_n+r_*u_n\in\partial S$, being a limit of ray points inside and
outside the closed set $S$; here $|z_n+su_n-q|=s-r_*<s$ directly. Second,
the contradiction scheme: in each case below we exhibit $s_0>0$,
depending only on the limit $z$ and the local constants there, such that,
for all large $n$, no $q\in\partial S$ with $|z_n-q|<2s$ can satisfy
$|z_n+su_n-q|<s$ when $0<s\le s_0$. Choosing $s\in(\delta_n,s_0]$ ---
possible as soon as $\delta_n<s_0$ --- this contradicts the first
observation; hence $\delta_n\ge s_0$ for all large $n$, contradicting
$\delta_n\to0$. Since the arcs are closed and meet only
at endpoints, and reentrant corners lie at distance $\ge\ell/2$ from $K$,
either $z$ lies in the relative interior of one arc, or $z$ is a vertex
$v$ with $\theta_v\le\pi$.

Case 1: $z$ in the relative interior of $\Gamma_i$. Choose $d>0$ so small
that $\bar B(z,4d)$ avoids the endpoints of $\Gamma_i$ and the other arcs,
and $\Gamma_i\cap\bar B(z,4d)$ is a connected subarc through $z$ (the rest
of $\Gamma_i$ is compact and does not contain $z$). Parametrizing this
subarc by arclength and Taylor-expanding, there is $C\ge1$ with
\begin{equation}\label{eq:graphest}
|\langle\nu(p), q-p\rangle|\le C |q-p|^2
\qquad\text{for }p,q\in\Gamma_i\cap\bar B(z,4d),
\end{equation}
$\nu$ the outer unit normal (the chord--arc comparison on a short $C^2$
subarc makes $C$ uniform). For large $n$, $z_n$ lies on this subarc, where
the unique proximal unit normal is $\nu(z_n)$: a unit $u$ with
$\langle u,\tau(z_n)\rangle\ne0$ is approached quadratically by boundary
points on the side it leans to, and $-\nu(z_n)$ points into
$\operatorname{int}S$; so $u_n=\nu(z_n)$. Set
$s_0:=\min\{d,\tfrac1{2C}\}$, take $0<s\le s_0$ and let $q\in\partial S$
satisfy $|z_n-q|<2s$ and $|z_n+s\nu(z_n)-q|<s$. For $n$ large,
$|z_n-z|<d$, so $|q-z|\le|q-z_n|+|z_n-z|<2s+d\le3d$ and $q$ lies on the
subarc; then \eqref{eq:graphest} gives
\[
|z_n+s\nu(z_n)-q|^2
= s^2-2s\langle\nu(z_n),q-z_n\rangle+|q-z_n|^2
\ge s^2+|q-z_n|^2(1-2Cs)\ge s^2,
\]
contradicting $|z_n+s\nu(z_n)-q|<s$; by the scheme above, $\delta_n\ge
s_0$ for all large $n$.

Case 2: $z=v$, a vertex with $\theta:=\theta_v\in(0,\pi]$. Let
$\Gamma_\pm$ be the incident arcs, $\tau_\pm$ their unit tangents at $v$
pointing away from $v$, and $\nu_\pm$ the outer normals at $v$, so that
$\langle\tau_+,\tau_-\rangle=\cos\theta$,
$\langle\nu_\pm,\tau_\mp\rangle=-\sin\theta\le0$ and
$\langle\nu_\pm,\tau_\pm\rangle=0$. Choose $d$ as in Case~1 (both subarcs
through $v$ connected in $\bar B(v,4d)$, no other arcs), and $C\ge1$ such
that \eqref{eq:graphest} holds along each subarc and, in arclength
parametrization from $v$, points of the subarcs and their normals expand
as
\[
p=v+a \tau_\pm+e_\pm(a),\qquad
\nu(p)=\nu_\pm+b_\pm(a) \tau_\pm+\eta_\pm(a),
\]
with $|e_\pm(a)|\le Ca^2$, $|b_\pm(a)|\le Ca$ and
$|\eta_\pm(a)|\le Ca^2$: the first-order variation of the unit
normal is tangent to the curve, and the $\nu_\pm$-component of the deviation is
$-\tfrac12|\nu(p)-\nu_\pm|^2=O(a^2)$ because both vectors are unit. For
large $n$, $z_n$ lies on one of the subarcs, say on the $+$ side, and
every obstruction at length $s\le s_0$ lies on one of the two subarcs,
since $|q-z_n|<2s$. We treat the position of $z_n$ in two subcases, with
$s_0$ and $d$ chosen below, small in terms of $C$ and $\theta$ only.

Subcase 2a: $z_n=v$. Here $u_n$ may be any direction of the normal fan,
but proximality forces $\langle u_n,\tau_\pm\rangle\le0$: if
$\langle u_n,\tau_+\rangle=c_0>0$, the point $q=v+a\tau_++e_+(a)$
with $a=sc_0$ satisfies $|v+su_n-q|^2\le s^2-s^2c_0^2+O(s^3)<s^2$ for
small $s$, so $\delta_S(v,u_n)=0$, contradicting $(v,u_n)\in N(S)$. Then,
for a competitor $q=v+a\tau_\pm+e_\pm(a)$ on either subarc ---
where $a\le2|q-v|<4s$ by the chord--arc comparison, after shrinking
$d$ ---
\[
\begin{aligned}
|v+su_n-q|^2
&=s^2-2sa \langle u_n,\tau_\pm\rangle
-2s\langle u_n,e_\pm(a)\rangle+|a\tau_\pm+e_\pm(a)|^2\\
&\ge s^2+a^2\bigl(1-2Cs-2Ca\bigr)\ge s^2
\end{aligned}
\]
for $s\le\tfrac1{32C}$ (so that also $a<4s\le\tfrac1{8C}$),
uniformly over the fan --- so no such $q$ exists.

Subcase 2b: $z_n=v+a_n\tau_++e_+(a_n)$ with $a_n>0$. As in
Case~1, $u_n=\nu(z_n)=\nu_++b_n\tau_++\eta_n$ with $|b_n|\le Ca_n$ and
$|\eta_n|\le Ca_n^2$. Obstructions on the same subarc are excluded by
the Case~1 computation. For an obstruction $q=v+a\tau_-+e_-(a)$ on
the other subarc (where $a\le2|q-v|\le4s+4a_n$ by the chord--arc
comparison),
\[
z_n+su_n-q=a_n' \tau_++s\nu_+-a \tau_-+R,
\qquad
a_n':=a_n+sb_n,\quad
R:=e_+(a_n)+s\eta_n-e_-(a),
\]
with $a_n'\in[\tfrac12a_n,\tfrac32a_n]$ for $s\le\tfrac1{2C}$
and $|R|\le2C(a_n^2+a^2)$. Writing $M:=a_n'\tau_++s\nu_+-a\tau_-$ and using
$\langle\tau_+,\tau_-\rangle=\cos\theta$,
$\langle\nu_+,\tau_+\rangle=0$,
$\langle\nu_+,\tau_-\rangle=-\sin\theta$,
\[
|M|^2=s^2+a_n'^2+a^2-2a_n'a\cos\theta+2sa\sin\theta
 \ge s^2+c_\theta (a_n'^2+a^2)
 \ge s^2+\tfrac{c_\theta}4(a_n^2+a^2),
\]
where $c_\theta:=\min\{1, 1-\cos\theta\}>0$: for $\cos\theta\ge0$ the
difference
$a_n'^2+a^2-2a_n'a\cos\theta-(1-\cos\theta)(a_n'^2+a^2)
=\cos\theta (a_n'-a)^2$ is nonnegative, for $\cos\theta<0$ the
cross term is itself nonnegative, and $2sa\sin\theta\ge0$ for
$\theta\in(0,\pi]$. Since $|M|\le s+a_n'+a\le6(s+a_n)$,
\[
|z_n+su_n-q|^2\ge|M|^2-2|M| |R|
\ge s^2+\Bigl(\tfrac{c_\theta}4-24C(s+a_n)\Bigr)(a_n^2+a^2)
\ge s^2
\]
once $s_0$ and $d$ are small enough that $24C(s+a_n)\le c_\theta/4$.
So again no such $q$ exists; by the scheme above, $\delta_n\ge s_0$ for
all large $n$, and the proposition follows.
\end{proof}

\subsection{Proof of Proposition~\ref{prop:polreachL}}\label{app:pfL}

\begin{proof}
Write $P(t):=8t+\bigl(\tfrac{5\pi}4-1\bigr)t^2$. All fibers of $L$ were
computed in Example~\ref{ex:Lshape}: the six units of outer boundary carry
infinite fibers, the fibers of the five convex fans are not cut before
length $1$, and the two reentrant walls, parametrized by the distance
$a\in(0,1)$ to the corner, have $\delta(a)=a$. Partitioning the tube by
feet and computing each piece exactly (the walls in the isometric
coordinates of Proposition~\ref{prop:cone}) gives, for $0<t<1$,
\[
V_L(t)=6t+2\int_0^1\min\{t,a\} da+\tfrac{5\pi}{4}t^2=P(t);
\]
hence $V_L=P$ on $[0,1]$ by continuity, and $\polreach(L)\ge1$. Let $t>1$ be sufficiently close
to $1$. Relative to the interval $(0,1)$, the feet decomposition of the
tube changes in exactly two ways.
First, the two reentrant walls saturate: their contribution is
$2\int_0^1\min\{t,a\} da=1$ instead of $2t-t^2$, which adds $(t-1)^2$ to
$P(t)$. Second, the fibers of the fans at $(2,1)$ and $(1,2)$ begin to cut
each other across the perpendicular bisector of the two corners: uncut,
each fan would cover a quarter disc of radius $t$, and the two quarter
discs now overlap. Nothing else changes: all remaining fibers are
infinite, and the quarter disc at $(2,1)$ lies in
$\{x\ge2, y\ge1\}$ and the one at $(1,2)$ in $\{x\ge1, y\ge2\}$, so they
meet the notch square, the side strips and the outer fans only in null
sets. Hence
\[
V_L(t) = 6t+1+\tfrac{5\pi}4t^2-|\mathcal O_t|
 = P(t)+(t-1)^2-|\mathcal O_t| ,
\]
where, placing the origin at $(2,2)$,
\[
\mathcal O_t:=\bigl\{(X,Y)\in[0,\infty)^2:\
X^2+(Y+1)^2\le t^2, (X+1)^2+Y^2\le t^2\bigr\}.
\]
The two circles cross on the diagonal at
$X=Y=a(t):=\tfrac12\bigl(\sqrt{2t^2-1}-1\bigr)$, so
\[
|\mathcal O_t|=\int_0^{a(t)}\Bigl(\sqrt{t^2-X^2}-1\Bigr) dX
+\int_{a(t)}^{t-1}\sqrt{t^2-(X+1)^2} dX .
\]
Write $t=1+s$. Then $a(1+s)=s-\tfrac{s^2}2+O(s^3)$. In the first integral,
$\sqrt{(1+s)^2-X^2}-1=s-\tfrac{X^2}2+O(sX^2)$ on $[0,a]$, so it equals
$sa-\tfrac{a^3}6+O(s^4)=s^2-\tfrac23s^3+O(s^4)$. In the second,
$t^2-(X+1)^2=(s-X)(2+s+X)=2(s-X)\bigl(1+O(s)\bigr)$ on $[a,s]$, so it
equals
$\sqrt2\cdot\tfrac23(s-a)^{3/2}\bigl(1+O(s)\bigr)
=\tfrac13s^3+O(s^4)$, using $s-a=\tfrac{s^2}2+O(s^3)$. Hence
\[
|\mathcal O_{1+s}|=s^2-\tfrac13s^3+O(s^4),
\qquad
V_L(1+s)-P(1+s)=\tfrac13s^3+O(s^4)>0
\]
for all small $s>0$: the quadratic corrections cancel exactly, and the
first deviation is cubic. If $V_L$ agreed with a polynomial $q$ of degree
at most two on $(0,R)$ for some $R>1$, then $q=P$, since the two
polynomials agree on $(0,1)$; this contradicts the cubic term. Hence
$\polreach(L)=1$.
\end{proof}

\subsection{Proofs for Section~\ref{sec:sharp}}\label{app:pfrest}

\begin{proof}[Proof of Proposition~\ref{prop:cusp}]
Feet dichotomy. Let $\Gamma_{\rm gap}$ be the union of the two open
gap walls, the floor segment $(0,1)\times\{0\}$ and the arc
$G^\circ:=\{(x,x^2):0<x<1\}$, and let
$\Gamma':=\partial S\setminus\Gamma_{\rm gap}$. Up to the null set of multi-feet
points, $V_S(t)=P(t)+F(t)$ with $P(t):=\vol(W_t(\Gamma'))$ and
$F(t):=\vol(W_t(\Gamma_{\rm gap}))$ --- a partition by the location of the foot,
with no ambient cut.

$P$ is a polynomial of degree $\le2$ on $(0,\rho_0)$, where $\rho_0:=\tfrac14$.
We claim $\delta_S\ge\rho_0$ on $N(S)\cap(\Gamma'\times\Sph^{1})$, splitting
$\Gamma'$ in two.

(i) Supported directions. If $z\in\partial S$, $u$ is a unit vector and
$S\subseteq\{w:\langle w-z,u\rangle\le0\}$, then for every $w\in S$,
$|z+su-w|^2=s^2-2s\langle u,w-z\rangle+|w-z|^2 \ge s^2+|w-z|^2,$
so $\dist(z+su,S)=s$ for all $s>0$ and $\delta_S(z,u)=\infty$; moreover the
inequality is strict for $w\ne z$. Since
$S\subseteq\{x\ge0\}\cap\{x\le1\}\cap\{y\ge-1\}\cap\{y\le1\}$, this applies to
every fiber over the four outer edges of $\Gamma'$, to every direction in the fan
at each of the three outer corners $(0,-1)$, $(1,-1)$, $(0,1)$ (a nonnegative
combination of the two adjacent outer normals), and to the fiber at the origin, whose only proximal direction is
$u=(-1,0)$: proximality being a local property, small $s$ suffices below. For
$u=(-a,b)$ with $b\ne0$ and all sufficiently small $s>0$, the left-edge point
$(0,sb)\in S$ is strictly closer to $su$ than the origin is; for $u=(a,b)$ with
$a>0$, $b\ne0$, and small $s>0$, the floor point $(sa,0)\in S$ gives
$|su-(sa,0)|=s|b|<s$; and for $u=(1,0)$ and $0<s<1$ the point $su=(s,0)$ itself
lies in $S$.

(ii) The mouth corners $(1,0)$ and $(1,1)$. Write $S=B_1\cup B_2$ with
$B_1:=[0,1]\times[-1,0]$ and $B_2:=\{(x,y): 0\le x\le1, x^2\le y\le1\}$; both
are convex ($B_2$ is an intersection of the convex sets $\{y\ge x^2\}$,
$\{y\le1\}$, $\{0\le x\le1\}$). Let $z$ be a mouth corner, contained in the block
$B$, and let $u\in\Nor(S,z)$ with $\delta:=\delta_S(z,u)<\infty$. Since $u$ is
proximal for $S$ at $z$, it is proximal for $B$ at $z$, so by convexity
$\langle w-z,u\rangle\le0$ for all $w\in B$, and by the strict inequality in (i)
no point of $B\setminus\{z\}$ can realize the distance from points of the fiber.
For $s>\delta$ we have $\dist(z+su,S)<s$ (Lemma~\ref{lem:deltameas}), so any foot
$w$ of $z+su$ lies in the other block $B'$ and satisfies
$|z-w|\le s+\dist(z+su,S)<2s$; letting $s\downarrow\delta$ gives
$\dist(z,B')\le2\delta$. Finally $\dist\bigl((1,1),B_1\bigr)=1$ and, for
$w=(x,y)\in B_2$,
\[
|w-(1,0)|^2=(1-x)^2+y^2 \ge (1-x)^2+x^4 \ge (1-x)^2+x-\tfrac12
=\bigl(x-\tfrac12\bigr)^2+\tfrac14 \ge \tfrac14
\]
(using $y\ge x^2$ and $x^4\ge x-\tfrac12$ on $[0,1]$, whose minimum
$4^{-4/3}-4^{-1/3}+\tfrac12>0$ is attained at $x=4^{-1/3}$), so
$\dist\bigl((1,0),B_2\bigr)\ge\tfrac12$. Hence $\delta\ge\tfrac14$ at both mouth
corners, proving the claim.

Consequently $\min\{t,\delta_S\}=t$ on $\eta':=N(S)\cap(\Gamma'\times\Sph^{1})$
for $0<t<\rho_0$, and \eqref{eq:HLW} gives
\[
P(t)=\Theta_1(\eta') t+\tfrac12 \Theta_0(\eta') t^{2} \in\R_2[t]
\qquad\text{on }(0,\rho_0),
\]
with finite coefficients as in Proposition~\ref{prop:regular} (the set is
$r$-bounded: compact feet, $\delta_S\ge\rho_0$).

$F$ is the covered gap area, up to $O(t^2)$. First, feet of points of $H$
lie on the open walls: the floor point $(x,0)$ strictly beats every point of
$\{1\}\times[-1,0]$ (whose distance is $\ge\sqrt{(1-x)^2+y^2}>y$); the vertical
competitor on the arc gives $\dist(p,G)\le x^2-y\le1-y\le|p-(1,1)|$, so the arc
beats the corner $(1,1)$; the minimizer on the arc is never the endpoint
$x^\ast=1$ (as $\varphi'(1)>0$, see Lemma~\ref{lem:cuspdepth}) nor $x^\ast=0$
(as $\varphi'(0)=-2a<0$), and the origin is never a foot of a gap point (the
floor is strictly closer). Second, fibers over the floor are vertical and fibers
over arc points with $x^\ast\le1-t$ increase the $x$-coordinate by at most their
length $\le t$, so both stay inside $H$ up to the cut. Third, for fixed $a$ the set of heights covered from the arc,
$\{y\in(0,a^2):\dist((a,y),G)\le t\}$, is an interval adjacent to the upper
wall: the function $y\mapsto\dist\bigl((a,y),\{(x,y):y\ge x^2\}\bigr)$ is
convex --- a distance to a convex set --- nonnegative, and vanishes at
$y=a^2$, hence nonincreasing on $(0,a^2]$; and on the range considered it
coincides with the distance to $G$, the nearest parabola points having
$0\le x^\ast\le1$. Consequently, writing
$\mathrm{covered}(a):=|\{y\in(0,a^2):\dist((a,y),S)\le t\}|
=\min\{a^2, t+\mathrm{par}(a,t)\}$, where
$\mathrm{par}(a,t):=\min\{a^2, \sup\{a^2-y: 0<y<a^2, \dist ((a,y),G)\le t \}\},$ $(\sup\emptyset:=0)$
is the height covered from the arc,
\[
F(t)=\int_0^{1-2t}\mathrm{covered}(a) da+O(t^2):
\]
the discrepancies (columns $a\in(1-2t,1)$, whose covered height is at most
$t+\sqrt{17} t$ by Lemma~\ref{lem:cuspcompare}, and fiber portions beyond the
mouth over arc points with $x^\ast>1-t$) are each contained in the
$t$-neighbourhood of a boundary piece of length $O(t)$, hence have area
$O(t^2)$: the $t$-neighbourhood of a rectifiable curve of length $\ell$ has
area at most $4\pi(\ell+2t)t$, by covering the curve with
$\lceil\ell/t\rceil+1$ balls of radius $2t$ centred at points $t$-dense
along it.

Expansion. For $a\le3\sqrt t$, Lemma~\ref{lem:cuspcompare} gives
$\min\{a^2,2t\}\le\mathrm{covered}(a)\le\min\{a^2,2t\}+(\sqrt{1+16a^2}-1)t
\le\min\{a^2,2t\}+72t^2$, so
\[
\int_0^{3\sqrt t}\mathrm{covered}(a) da
=\Bigl(6-2\sqrt2+\tfrac{2\sqrt2}{3}\Bigr)t^{3/2}+O(t^{5/2}).
\]
For $3\sqrt t<a\le1-2t$: by Lemma~\ref{lem:cuspcompare},
$\mathrm{par}(a,t)\le t\sqrt{1+16a^2}\le\sqrt{17} t<a^2$, so the supremum
defining $\mathrm{par}$ is attained at an interior point of the column at
distance exactly $t$ from $G$ (continuity of the distance), and
Lemma~\ref{lem:cuspdepth} with $h=t$ gives
$\mathrm{par}(a,t)=t\sqrt{1+4a^2}+O(t^2)$ uniformly; moreover
$t+\mathrm{par}\le(1+\sqrt{17})t<a^2$, so the minimum is the second argument, and
\[
\int_{3\sqrt t}^{1-2t}\bigl(t+\mathrm{par}(a,t)\bigr) da
=t\int_0^{1}\bigl(1+\sqrt{1+4a^2}\bigr) da-6 t^{3/2}+O(t^{2}),
\]
since $\int_0^{3\sqrt t}(1+\sqrt{1+4a^2}) da=6\sqrt t+O(t^{3/2})$ and the
trimmed sliver contributes $t\int_{1-2t}^{1}(\cdot) da=O(t^2)$. Adding the two
ranges,
\[
F(t)=A t-\frac{4\sqrt2}{3} t^{3/2}+O(t^2),
\qquad A:=\int_0^{1}\bigl(1+\sqrt{1+4a^2}\bigr) da .
\]

Conclusion. $V_S=P+F=q(t)-\tfrac{4\sqrt2}{3}t^{3/2}+O(t^2)$ with
$q\in\R_2[t]$. If $V_S$ agreed with a polynomial $p$ on some $(0,\varepsilon)$,
then $(p-q)(t)/t^{3/2}\to-\tfrac{4\sqrt2}{3}$; but for a polynomial $p-q$
vanishing at $0$ this quotient tends to $0$ or $\pm\infty$. Contradiction.
\end{proof}

\begin{proof}[Proof of Proposition~\ref{prop:twodiscs}]
Write $A(r):=2\pi r^2-2r^2\arccos(h/r)+2h\sqrt{r^2-h^2}$, so that
$\vol(S_t)=A(R+t)$; $A$ is real-analytic on $(h,\infty)$. If $V_S$ agreed with a
polynomial $p$ on some $(0,\varepsilon)$, then $A(r)=p(r-R)+A(R)$ on
$(R,R+\varepsilon)$, hence on all of $(h,\infty)$ by the identity theorem for
real-analytic functions. We show $A$ has a nonzero $(r-h)^{3/2}$-term at $r=h$, which
no polynomial has. Put $u:=r-h\downarrow0$. Using the standard expansions of
$\arccos(1-x)$ and $\sqrt{1+x}$, a direct computation gives (the
$u^{1/2}$-terms produced by the two terms of the lens area cancel)
$A(h+u)=2\pi(h+u)^2-\frac{8\sqrt2}{3} h^{1/2} u^{3/2}+O(u^{5/2}),$
consistent with the classical fact that the lens of two just-overlapping discs is the
union of two circular segments of sagitta $u$, each of area
$\tfrac{4\sqrt2}{3}\sqrt{h} u^{3/2}(1+o(1))$. Since
$-\tfrac{8\sqrt2}{3}h^{1/2}\ne0$, $A$ is not a polynomial on $(h,\infty)$:
contradiction.
\end{proof}

\begin{proof}[Proof of Proposition~\ref{prop:cantor}]
Two preliminaries: every gap of $K$ is bounded
by two rays of the fan, and gaps of arbitrarily small angular width $g$ occur; the
fiber over a point at radius $r$ on a ray bounding a gap of width $g$, in the
direction of that gap, is cut by the gap's bisector at length $r\tan(g/2)$.
Suppose $S$ admitted a conic certificate of radius $\rho>0$ with data
$\Sigma$, $(r_\sigma)$, $(C_\sigma)$.

Case A: every $\sigma\in\Sigma$ has $|\sigma|>2r_\sigma$. Set
$r^\ast:=\min\bigl\{\tfrac12, \min_\sigma(|\sigma|-r_\sigma)\bigr\}>0$,
with $\min_\emptyset:=+\infty$ (so $r^\ast=\tfrac12$ when
$\Sigma=\emptyset$).
Every point $z$ of $S$ at radius $r^\ast$ lies outside all cores:
$|z-\sigma|\ge|\sigma|-r^\ast\ge r_\sigma$. So (C2) applies at $z$;
choosing a gap with $r^\ast\tan(g/2)<\rho$ contradicts it.

Case B: some $\sigma\in\Sigma$ has $|\sigma|\le2r_\sigma$. By (C1),
$S\cap\bar B(\sigma,3r_\sigma)=(\sigma+C_\sigma)\cap\bar B(\sigma,3r_\sigma)$.

If $\sigma=0$: first $3r_\sigma\le1$, since otherwise picking $\omega\in K$
gives $\omega\in C_0$, whence
$\tfrac{1+3r_\sigma}{2} \omega\in C_0\cap\bar B(0,3r_\sigma)=
S\cap\bar B(0,3r_\sigma)$, contradicting the truncation at radius $1$. Then
$C_0$ coincides with the Cantor fan cone on $\bar B(0,3r_\sigma)$, and
since a cone is determined by its germ at the apex, $C_0$ is the Cantor
fan cone itself --- whose link has gaps of arbitrarily small width,
contradicting Theorem~\ref{thm:link}, which forces every gap to have width
at least $2\arctan(\rho/r_\sigma)>0$.

If $\sigma\ne0$: write $\sigma=|\sigma|e^{i\phi_0}$ with $\phi_0\in K$,
$0<|\sigma|\le1$ (as $\Sigma\subset\partial S$). Pick
$\phi_1\in K\setminus\{\phi_0\}$ and, for
$0<\varepsilon\le\tfrac23 r_\sigma$,
$z_\varepsilon:=\varepsilon e^{i\phi_1}\in S\cap\bar B(\sigma,3r_\sigma)$
(as $|z_\varepsilon-\sigma|\le\varepsilon+2r_\sigma<3r_\sigma$). Then
$z_\varepsilon-\sigma\in C_\sigma$, so with $\lambda:=\tfrac98$ the point
$w:=\sigma+\lambda(z_\varepsilon-\sigma)$ belongs to $\sigma+C_\sigma$, and
$|w-\sigma|=\lambda|z_\varepsilon-\sigma|\le\tfrac98(\varepsilon+2r_\sigma)
\le3r_\sigma$, so $w\in S$. But
$w=-(\lambda-1)|\sigma|e^{i\phi_0}+\lambda\varepsilon e^{i\phi_1}$: letting
$\varepsilon\downarrow0$, $w$ has positive radius and direction converging
to $-e^{i\phi_0}$, whose angle lies outside $[0,\tfrac\pi2]\supseteq K$;
for small $\varepsilon$ this contradicts $w\in S$. Hence no radius is
certifiable: $\conreach(S)=0$. (Equivalently, in the conic zone the
normalized fiber lengths $\tilde\delta$ are not bounded below, in line with
Theorem~\ref{thm:link}.)

The parallel volume admits a bilateral estimate: there are $0<c_1\le c_2$ with
\[
c_1 t^{1-D} \le V_S(t) \le c_2 t^{1-D},\qquad 0<t\le t_0 .
\]
We first compare tube sections with linear neighbourhoods of $K$: write
$K_\varepsilon:=\{\phi\in[0,\pi]: \dist(\phi,K)\le\varepsilon\}$
(Euclidean distance on $[0,\pi]$) and, for
$0<r\le1$, $A_r:=\{\phi\in[0,\pi]: re^{i\phi}\in S_t\}$. We claim
$K_{t/r} \subseteq A_r \subseteq K_{(\pi/2) t/r}$ for $2t\le r\le1 .$
If $\dist(\phi,K)\le t/r$, pick $\psi\in K$ with $|\phi-\psi|\le t/r$; then
$\dist(re^{i\phi},S)\le|re^{i\phi}-re^{i\psi}|\le r|\phi-\psi|\le t$. Conversely
let $re^{i\phi}\in S_t$; its nearest point of $S$ lies on some segment
$\{\rho e^{i\psi}:0\le\rho\le1\}$, $\psi\in K$, and the distance to that segment
is at least the distance to the full ray, which is
$r\sin\bigl(\min\{|\phi-\psi|,\tfrac\pi2\}\bigr)\ge
\tfrac{2r}{\pi}\min\{|\phi-\psi|,\tfrac\pi2\}$; since this is $\le t\le r/2$,
the minimum is not $\tfrac\pi2$ and $|\phi-\psi|\le\tfrac\pi2 \tfrac tr$.

For the ternary Cantor set, $\lambda_1(K_\varepsilon)\asymp\varepsilon^{1-D}$
for $0<\varepsilon\le1$ (see e.g. \cite[Ch.~2]{Falconer} or \cite{LvF}). In
polar coordinates the annulus $\{\tfrac12\le|y|\le1\}$ contributes
$\int_{1/2}^1 r \lambda_1(A_r) dr\ge\int_{1/2}^1
r \lambda_1(K_{t/r}) dr\gtrsim t^{1-D}$, giving the lower bound; for the upper
bound, the disc $B(0,2t)$ contributes $O(t^2)$; the radii $r\in(2t,1]$ contribute
$\int_{2t}^{1}r \lambda_1(K_{(\pi/2)t/r}) dr\lesssim
\int_{2t}^{1}r (t/r)^{1-D} dr=O(t^{1-D})$, where tube points at radius $r$
with angle outside $[0,\pi]$ (slightly negative angles near the ray $\phi=0$)
are not counted in $A_r$ but have angular width $O(t/r)$ by the same
computation, contributing $\int_{2t}^1 r O(t/r) dr=O(t)=o(t^{1-D})$; and the
outer rim: if $|y|>1$ and $\dist(y,S)\le t$, its nearest point
$\rho e^{i\psi}\in S$ has $\rho\ge|y|-t>1-t$, so
$|\rho e^{i\psi}-e^{i\psi}|=1-\rho\le t$ and $y$ lies in the
$2t$-neighbourhood of $E:=\{e^{i\phi}:\phi\in K\}$; since $\phi\mapsto e^{i\phi}$
is bi-Lipschitz on $[0,\pi/2]$, the two-sided estimate
$\lambda_1(K_\varepsilon)\asymp\varepsilon^{1-D}$ transfers: the covering
numbers of $E$ at scale $\varepsilon$ are of order $\varepsilon^{-D}$, whence
$\lambda_2(E_{2t})=O(t^{2-D})$. Since
$1-D<\min\{2-D,2\}$ the annulus dominates.

Since $1-D<1$, $V_S(t)/t\to\infty$, while any polynomial $p$ with $p(0)=0$ has
$p(t)/t\to p'(0)<\infty$: hence $V_S$ agrees with no polynomial near $0$, the
outer Minkowski content is infinite, and $\polreach(S)=0$.
\end{proof}

\end{document}